\pdfoutput=1
\RequirePackage{ifpdf}
\ifpdf 
\documentclass[pdftex]{sigma}
\else
\documentclass{sigma}
\fi

\numberwithin{equation}{section}

\newtheorem{Theorem}{Theorem}[section]
\newtheorem*{Theorem*}{Theorem}

\newtheorem{Proposition}[Theorem]{Proposition}

\theoremstyle{definition}
\newtheorem{Definition}[Theorem]{Definition}

\newtheorem{Remark}[Theorem]{Remark}

\def\nn{\notag}
\def\Z2{\mathbb{Z}_2^2}
\def\g{\mathfrak{g}}

\def\osp{\mathfrak{osp}}
\def\gl{\mathfrak{gl}}
\def\sl{\mathfrak{sl}}

\def\qq{$\Z2$-$\q(n)$}
\def\q{\mathfrak{q}}
\def\so{\mathfrak{so}}
\def\sp{\mathfrak{sp}}
\def\Zsl{$\mathbb{Z}_3^2$-$\sl(2)$}
\def\Zosp{$\Z2$-$\osp(m|2n)$}

\usepackage{bm}%
\usepackage{stmaryrd}
\usepackage{nicematrix}
\usepackage{arydshln}
\usepackage{braket}

\begin{document}

\allowdisplaybreaks

\newcommand{\arXivNumber}{2604.08900}

\renewcommand{\PaperNumber}{070}

\FirstPageHeading

\ShortArticleName{Graded Casimir Elements and Central Extensions of Color Lie Algebras}

\ArticleName{Graded Casimir Elements and Central Extensions\\ of Color Lie Algebras}

\Author{Naruhiko AIZAWA~$^{\rm a}$, Ichi FUJII~$^{\rm a}$, Jambulingam SEGAR~$^{\rm b}$ and Joris VAN DER JEUGT~$^{\rm c}$}

\AuthorNameForHeading{N.~Aizawa, I.~Fujii, J.~Segar and J.~Van der Jeugt}

\Address{$^{\rm a)}$~Department of Physics, Graduate School of Science, Osaka Metropolitan University,\\
\hphantom{$^{\rm a)}$}~Sugimoto Campus, Osaka 558-8585, Japan}
\EmailD{\mail{aizawa@omu.ac.jp}, \mail{sq25482d@st.omu.ac.jp}}

\Address{$^{\rm b)}$~VICAS, Ramakrishna Mission Vivekananda College, Chennai-600 004, India}
\EmailD{\mail{segar@rkmvc.ac.in}}

\Address{$^{\rm c)}$~Department of Mathematics, Computer Science and Statistics, 	Ghent University,	\\
\hphantom{$^{\rm c)}$}~Krijgslaan 281-S9, B-9000 Gent, Belgium}
\EmailD{\mail{Joris.VanderJeugt@UGent.be}}

\ArticleDates{Received April 25, 2026, in final form July 30, 2026; Published online August 08, 2026}

\Abstract{A color Lie algebra is a generalization of a Lie (super)algebra by an Abelian group $\Gamma$. The underlying vector space and defining relations of the algebra are graded by $\Gamma$, and a color Lie algebra can admit graded Casimir elements. Furthermore, in that case its loop algebra admits graded central extensions. 	We present a general method for constructing 2nd order graded Casimir elements and graded central extensions for a given color Lie algebra and its loop algebra, respectively. We also show that there exists a large class of color Lie algebras admitting such graded Casimir elements or central extensions by providing three examples, namely, $\mathfrak{sl}(2)$ for $\Gamma = \mathbb{Z}_3^2$, and $\mathfrak{q}(n)$ and $\mathfrak{osp}(m|2n)$ for $\Gamma = \mathbb{Z}_2^2$.}


\Classification{17B75}

\section{Introduction}

The term color Lie algebra is loosely used, following Rittenberg and Wyler \cite{rw2}, to describe a~generalization of Lie algebras defined by a pair $(\Gamma,\omega)$, where $\Gamma $ is an additive Abelian group and $\omega\colon \Gamma \times \Gamma \to \mathbb{C}^*$ is a mapping called commutative factor (see Section~\ref{SEC:Pre} for a precise definition). This algebraic structure has a long history. It was first introduced by Rimhak Ree in 1960 \cite{Ree}, and later rediscovered by Rittenberg and Wyler \cite{rw1,rw2}. Shortly after this rediscovery, color Lie algebras were extensively studied by Scheunert \cite{scheu,scheuCasi,scheuTen}.

In fact, Lie superalgebras constitute the simplest example of color Lie algebras, corresponding to the case where $\Gamma = \mathbb{Z}_2 = \{ 0, 1 \}$ and $ \omega(\alpha,\beta) =(-1)^{\alpha\beta}$ for $ \alpha,\beta \in \mathbb{Z}_2$.
A striking feature of color Lie algebras is their $\Gamma$-grading structure.
In the case of Lie superalgebras, this means the following: $\Gamma$ decomposes the superalgebra into two subspaces, namely the even ($0$-graded) and odd sectors ($1$-graded).
Furthermore, the defining relations are given, in terms of these sectors, by the following (anti)commutation relations:
\begin{equation*}
	[\mathrm{even}, \mathrm{even}] \subseteq \mathrm{even}, \qquad
	[\mathrm{even},\mathrm{odd}] \subseteq \mathrm{odd}, \qquad
	\{ \mathrm{odd}, \mathrm{odd}\} \subseteq \mathrm{even}.
\end{equation*}
One may see that these relations are compatible with the multiplication law of $\mathbb{Z}_2$.

Another simple but widely studied example is obtained by taking $\Gamma = \Z2 := \mathbb{Z}_2 \times \mathbb{Z}_2$. There are two possible choices of $\omega$ for this $\Gamma$, and the corresponding color Lie algebras consist of four (the order of $\Z2$) sectors, each of which is graded by an element of $\Z2$~\cite{rw2}.
The (anti)commutation relations of these sectors are compatible with the multiplication law of~$\Z2$.
In this case, one can observe many new features which are distinct from those of Lie superalgebras.
One such feature is the \textit{graded} center, namely the set of elements that commute (in the graded sense) with all elements of a color Lie algebra.

For a given finite-dimensional color Lie algebra, one may define its universal enveloping algebra and its graded center.
Elements of the graded center are the graded Casimir elements of the color Lie algebra \cite{GrJav, scheu,scheuCasi,scheuTen}.
One may also consider a loop extension of a given finite-dimensional color Lie algebra, which can admit an extension by graded central elements.
Such graded Casimir elements and central extensions are found in $\Z2$-graded color extensions of~$\sl(2)$ and~$\osp(1|2)$, and a general method to construct them for $\Z2$ is developed in \cite{AFITTsuper, AIKTTslint} (see also Appendix of~\cite{AiKimura}).
The graded central extension is also found for a module extension of the Virasoro algebra with $\Gamma = \Z2$ grading~\cite{AizawaConfType}.

The purpose of this paper is to extend the general method for $\Gamma = \Z2$ to arbitrary Abelian groups and to show, by presenting explicit examples, that there exist a large number of color Lie algebras possessing such graded Casimir elements and graded central extensions.
The examples considered in this paper include the $\mathbb{Z}_3^2$-graded color extension of $\sl(2)$ and the $\Z2$-graded color extensions of the Lie superalgebras~$\q(n)$ and~$\osp(m|2n)$. These examples share the property that they contain copies of the underlying Lie (super)algebra as subalgebras.

This work is motivated by recent advances and applications of color Lie algebras in various fields.
If a color Lie algebra appears in a system under consideration, it is important to determine whether it admits graded Casimir elements (or graded central elements in the case of affine color Lie algebras).
Focusing only on ordinary (non-graded) Casimir elements fails to capture the full structure of the system.

Before starting our analysis, we provide a concise state-of-the-art overview of studies on color Lie algebras in mathematics and physics.
Apart from studies of structure and representation theory (see, for example, \cite{IsStVdJ,KuToclassification,SilvCoho,RyanRefining,SchZhang,Stclassical,StG2}), color Lie algebras have also been explored in the following areas of mathematics and physics:
\begin{enumerate}\itemsep=0pt
\item[(i)] \emph{Noncommutative geometry} (see \cite{PonSch} for a review).
Color Lie algebras defined for $ \Gamma = \mathbb{Z}_2^n := \mathbb{Z}_2 \times \mathbb{Z}_2 \times \cdots \times \mathbb{Z}_2 $ ($n$ times) provide a highly nontrivial example of a noncommutative manifold. The study of geometry on such spaces, initiated in \cite{Covolo,CoGrPo,CoGrPo2012}, has since been extensively developed \cite{BruIb,BruPon,bruce2019products,CoGrPo2016,CoKwoPo,MoVar}.
However, many open problems still remain. One of the most significant issues is the definition of integration on these noncommutative spaces~\cite{Poncin}.
So far, three different proposals for such a definition have been put forward \cite{AiItoIntegral, Poncin,PonSch}, but no consensus has yet been reached.

\item[(ii)] \emph{Integrable systems}.
Recalling that many classical integrable systems, such as soliton equations—and in particular Toda field theory—can be obtained in an algebraic way (see, e.g.,~\cite{BBTbook}), it is natural to expect that the use of color Lie algebras can produce novel types of integrable systems.
This has indeed been realized for $\Gamma = \Z2$: a generalization of the sine-Gordon equation \cite{BruSG}, the Liouville and sinh-Gordon equations \cite{AFITTsuper, AIKTTslint}, as well as the mKdV and KdV equations \cite{AiFujiIto}, have been obtained.

\item[(iii)] \emph{Knot theory}.
It was shown in \cite{AiKimura} that a color Lie algebra admitting a nondegenerate invariant bilinear form gives rise to a universal weight system (UWS) that encodes the first nontrivial layers of finite-type topological invariants of knots.
In general, a UWS is constructed in terms of the Casimir elements of the algebra under consideration.

\item[(iv)] \emph{Generalization of supersymmetry}.
Supersymmetry is a physical realization of the super-Poincar\'e algebra.
Since color Lie algebras generalize Lie superalgebras, one may consider color generalizations of the super-Poincar\'e algebra.
The case $\Gamma = \mathbb{Z}_4 \times \mathbb{Z}_4$ was considered by Wills-Toro \cite{WTIQ,WTz4}, while the cases $\Gamma = \Z2$ and $\Gamma = \mathbb{Z}_2^n$ were introduced by Tolstoy \cite{tolstoy2014} and Bruce \cite{BruceSUSY}, respectively.
Bruce’s algebras are now widely accepted, and generalizations of supersymmetric quantum mechanics \cite{Aizawa_2020,AiAmaDoiZn, BruDup} as well as two-dimensional field theories have been constructed \cite{AiItoTa, BruSigma}.
A~remarkable observation is that the $S$-matrix of quantum field theory admits Bruce's algebra as symmetries, providing an extension of the well-known Haag--{\L}opusza{\'n}ski--Sohnius theorem \cite{ItoNago}.

\item[(v)] P\emph{arastatistics}.
Particles beyond bosons and fermions (paraparticles) have recently attracted renewed attention from both theoretical \cite{wang2025particle} and experimental perspectives \cite{huerta2025particle,alderete2025experimental}.
Color Lie algebras provide a natural framework for describing paraparticles through their representation theory \cite{BALBINO,SVdJ2018,SVdJ2024,toppan2022first,toppan2024braid}.
An important observation is that such paraparticles are, in principle, theoretically detectable; that is, there exist quantum operators that distinguish paraparticles from ordinary bosons and fermions \cite{toppan2021inequivalent,toppanMulti,toppan2024detectability}.
It was also shown that a mixed system of well-known parabosons and parafermions, introduced by Green, forms a color Lie algebra with \smash{$\Gamma = \Z2$} \cite{tolstoyOnce}.

\item[(vi)] \emph{Hidden color Lie algebra structure}.
There exist systems that are formulated without any relation to color Lie algebras. However, their symmetries can be enhanced to form a color Lie algebra, or they can be solved using a color Lie algebra.
As examples, we mention de Sitter supergravity with a positive cosmological constant \cite{Vasiliev}, the L\'evy–Leblond equation \cite{LLEprogress,LLEproceed,RyanLL}, and the Dirac oscillator \cite{IsRyan}.
\end{enumerate}

This paper is organized as follows: Section~\ref{SEC:Pre} collects definitions and basic facts on color Lie algebras that will be used in the subsequent sections.
In Section~\ref{SEC:Method}, after introducing the definition of the graded center of a color Lie algebra, we formulate a general method for constructing a~graded bilinear form and 2nd order graded Casimir elements. We then consider graded central extensions of the loop algebra of a color Lie algebra, and show that these extensions are also determined by the graded bilinear form.
The rest of this paper is devoted to applications of the method developed in Section~\ref{SEC:Method}.
We present three examples of color Lie algebras that admit graded Casimir elements.
Their loop algebras admit graded central extensions.
In Section~\ref{SEC:Eg1}, the $\Gamma = \Z2$ color extension of the Lie superalgebra $\q(n)$ and the $\Gamma = \mathbb{Z}^2_3$ extension of the Lie algebra $\sl(2)$ are investigated.
In Section~\ref{SEC:osp}, a~more involved case, namely the Lie superalgebra $\osp(m|2n)$ for $\Gamma = \Z2$, is investigated in detail.
Finally, we summarize the present work and provide some concluding remarks in Section~\ref{SEC:Conclusion}.

\section{Preliminaries} \label{SEC:Pre}

This section presents the basic notions of color Lie algebras that will be used in this paper.
We give only the minimum necessary for the present study.
The reader may refer to \cite{scheu,scheuTen,zhang2025} for further details.

Let $\Gamma$ be an additive Abelian group. We fix $\Gamma$ and suppose that all vector spaces and algebras are defined over the field $\mathbb{C}$.
A vector space $V$ is said to be $\Gamma$-graded if it can be written as a~direct sum of subspaces indexed by elements of $\Gamma$:
\begin{equation*}
	V = \bigoplus_{\gamma \in \Gamma} V_{\gamma}.
\end{equation*}
We say that an element of $ V_{\gamma}$ has \textit{degree} $\gamma$.
A linear mapping $ f\colon V \to W $ between two $\Gamma$-graded vectors spaces $V$ and $ W $ is said to be of degree $\gamma \in \Gamma $ if
\begin{equation*}
	f(V_{\alpha}) \subset W_{\alpha+\gamma}, \qquad \forall \alpha \in \Gamma.
\end{equation*}
The vector space $ \operatorname{Hom}	(V,W)$ of all linear mappings of $ V $ into $ W $ is naturally $\Gamma$-graded.

An algebra $S$ is called $ \Gamma$-graded if its underlying vector space is $\Gamma$-graded and if it satisfies
\begin{equation*}
	S_{\alpha} S_{\beta} \subset S_{\alpha+\beta}, \qquad \alpha, \beta \in \Gamma.
\end{equation*}
To define color Lie algebras ($\Gamma$-graded $\omega$-Lie algebras), we introduce a mapping called a \textit{commutative factor} on $\Gamma$ (called commutation factor in \cite{scheu}) $ \omega \colon \Gamma \times \Gamma \to \mathbb{C}^{\ast}$, which satisfies
\begin{align}
	&\omega(\alpha,\beta)= \omega(\beta,\alpha)^{-1},
	\nn \\
	&\omega(\alpha,\beta+\gamma)= \omega(\alpha,\beta) \omega(\alpha,\gamma),
	\nn \\
	&\omega(\alpha+\beta,\gamma)= \omega(\alpha,\gamma) \omega(\beta,\gamma), \qquad
	\forall \alpha, \beta, \gamma.
	\label{omegaprop1}
\end{align}
These relations imply
\begin{align}
	&\omega(0,0)= \omega(0,\alpha) = \omega(\alpha,0) =1, \qquad \omega(\alpha,\alpha) = \pm 1,
	\nn\\
	&\omega(\alpha,-\beta)= \omega(-\alpha,\beta) = \omega(\beta,\alpha), \qquad
	\omega(-\alpha,-\beta) = \omega(\alpha,\beta), \qquad \forall \alpha, \beta.
	\label{omegaprop2}
\end{align}

\begin{Definition}
	Let \smash{$\g = \bigoplus_{\gamma\in\Gamma}\g_{\gamma}$} be a $\Gamma$-graded vector space and let $\omega$ be a commutative factor on $\Gamma.$
	Then $\g$ is called a color Lie algebra, if $\g$ admits a bilinear map
	$ \llbracket\ , \ \rrbracket\colon \g \times \g \to \g, $ which is of degree $0$, satisfying the following conditions:
	\begin{align}
		&\llbracket X,Y \rrbracket= -\omega(\alpha,\beta)\llbracket Y,X\rrbracket,
		\nonumber\\
		&\llbracket X, \llbracket Y, Z \rrbracket \rrbracket= \llbracket \llbracket X, Y \rrbracket, Z \rrbracket + \omega(\alpha,\beta) \llbracket Y, \llbracket X, Z \rrbracket \rrbracket,
		\qquad \forall X, Y, Z \in \g,
		\label{colorJacobi}
	\end{align}
	where $ \alpha = \deg(X)$, $\beta = \deg(Y)$ are degrees of the elements $ X$, $Y$, respectively.
\end{Definition}	

We call the bilinear map $ \llbracket\ , \ \rrbracket $ a color Lie bracket.
In some literature, a distinction is made between color Lie algebras and color Lie superalgebras.
The difference lies in the sign of $ \omega(\alpha,\alpha)$.
If $ \omega(\alpha,\alpha) = 1 $ for all $ \alpha \in \Gamma,$ then $\g$ is called a color Lie algebra. If there exists at least one $\alpha \in \Gamma$ such that $\omega(\alpha,\alpha) = -1$, then $\g$ is called a color Lie superalgebra. We do not make this distinction in this paper.

Any $\Gamma$-graded associative algebra $ S$ may be turned into a color Lie algebra by defining the color Lie bracket as
\begin{equation}
	\llbracket X, Y \rrbracket = XY - \omega(\alpha,\beta) YX, \qquad
	X, Y \in S, \quad \alpha = \deg(X),\quad \beta = \deg(Y)
	\label{ColorBracket}
\end{equation}
for any commutative factor $\omega.$
In the case of $ \omega(\alpha,\beta) = +1 $ or $ -1$ for all $\alpha, \beta \in \Gamma, $ the color Lie bracket is a commutator or an anticommutator.

Like for Lie algebras and Lie superalgebras, $ \operatorname{End}(V) := \operatorname{Hom}(V,V)$ on a $\Gamma$-graded vector space $V$ provides an example of color Lie algebras.
Observe that $\operatorname{End}(V)$ is an associative $\Gamma$\nobreakdash-graded algebra with the composition of endomorphisms as the multiplication.
Then, \textit{the general linear color Lie algebra} $ \gl(V,\omega)$ is defined as $\operatorname{End}(V)$ equipped with the color Lie bracket given by~\eqref{ColorBracket}.
To be more explicit, we present a basis of $\gl(V,\omega)$ in terms of matrices.
Suppose that $ \alpha < \beta < \gamma < \cdots $ is a fixed total order for $\Gamma.$
Choose an ordered basis for each graded subspace of $V$ and order the basis of $ V_{\alpha}, V_{\beta}, V_{\gamma}, \dots $ according to the total order for $\Gamma.$
This yields an ordered basis of~$V$.
Then an element $E(\alpha,\beta) \in \operatorname{Hom}(V_{\beta}, V_{\alpha})$ is represented by a block matrix of the following form:
\begin{equation}
	E(\alpha,\beta)=
	\bordermatrix{ & \alpha & \beta & \gamma & \cdots \cr
		\alpha & \bm{\cdot} & \ast & \bm{\cdot} & \ldots \cr
		\beta & \bm{\cdot} & \bm{\cdot} & \bm{\cdot} \cr
		\gamma & \bm{\cdot} & \bm{\cdot} & \bm{\cdot} & \ldots \cr
		\vdots & \vdots & \vdots & \vdots & \cr
	},
	\label{Edef}
\end{equation}
where the dot $\bm{\cdot}$ denotes a zero matrix and the asterisk $\ast$ denotes the block containing nonzero entries.
We denote the matrix unit relative to the ordered basis of $V$ by $ E(\alpha,\beta)_{ij} $ where $ 1 \leq i \leq \dim V_{\alpha}$, $1 \leq j \leq \dim V_{\beta}$. It is the matrix~\eqref{Edef} whose only nonzero entry is $1$ at the $(i,j)$-position of the asterisk block.
It is easy to see the following:
\begin{align*}
	&\deg(E(\alpha,\beta)_{ij})= \alpha - \beta,
	\qquad
	E(\alpha,\beta)_{ij} E(\gamma,\delta)_{k\ell}= \delta_{\beta\gamma} \delta_{jk} E(\alpha,\delta)_{i\ell}.
\end{align*}
It follows that the matrix units form a basis of $\gl(V,\omega)$ and we can express the defining relations of $\gl(V,\omega)$ as
\begin{equation*}
	\llbracket E(\alpha,\beta)_{ij}, E(\gamma,\delta)_{k\ell} \rrbracket =
	\delta_{\beta\gamma} \delta_{jk} E(\alpha,\delta)_{i\ell} - \omega(\alpha-\beta,\gamma-\delta) \delta_{\alpha\delta} \delta_{\ell i} E(\gamma,\beta)_{kj}.
\end{equation*}

The degree $\alpha$ subspace of $\gl(V,\omega)$ is
$ \gl(V,\omega)_{\alpha} := \sum_{\beta \in \Gamma} \operatorname{Hom}(V_{\beta}, V_{\beta+\alpha}).$
An element $ X(\alpha,\beta) \in \gl(V,\omega)_{\alpha-\beta}$ is a linear combination of $E(\alpha,\beta)_{ij}$:
\begin{equation*}
	X(\alpha,\beta) = \sum_{i=1}^{\dim V_\alpha} \sum_{j=1}^{\dim V_\beta} x_{ij} E(\alpha,\beta)_{ij}, \qquad x_{ij} \in \mathbb{C},
\end{equation*}
and an element $ X \in \gl(V,\omega)$ can be written as
\begin{equation*}
	X = \sum_{\alpha,\beta\in \Gamma} X(\alpha,\beta).
\end{equation*}
We now define a generalization of the trace.

\begin{Definition}
	The color trace $\operatorname{ctr}\colon \operatorname{End}(V) \to \mathbb{C}$ is defined by
	\begin{equation*}
		\operatorname{ctr} X = \sum_{\alpha \in \Gamma} \omega(\alpha,\alpha) \operatorname{tr} X(\alpha,\alpha),
	\end{equation*}
	where $\operatorname{tr}$ denotes the usual trace of a matrix. Clearly, the color trace is a mapping of degree zero.
\end{Definition}
\begin{Proposition}[\cite{scheuTen,zhang2025}] We have
	\begin{equation}
		\operatorname{ctr}(X(\alpha) X(\beta)) = \omega(\alpha,\beta) \operatorname{ctr}(X(\beta) X(\alpha)),
		\label{ctrperm}
	\end{equation}
	where $ X(\alpha)$ and $ X(\beta) $ are homogeneous elements of $\gl(V,\omega)$ of degree $\alpha$ and $\beta$, respectively.
\end{Proposition}

Obviously,
$ \sl(V,\omega) := \{ X \in \gl(V,\omega)\mid \operatorname{ctr} X = 0 \} $
is a $\Gamma$-graded subalgebra of $\gl(V,\omega). $
To introduce $\Gamma$-graded subalgebras of $\gl(V,\omega)$, we introduce a bilinear form on~$V$.
Let $ J\colon V \times V \to \mathbb{C}$ be a degree-zero and nondegenerate bilinear form on~$V$.
The condition $\deg(J) = 0 $ implies that for any $ x, y \in V$
\begin{equation*}
	J(x,y) = 0 \qquad \text{if} \ \deg(x) + \deg(y) \neq 0.
\end{equation*}
The nondegeneracy of $J$ requires that $ \dim V_{\alpha} = \dim V_{-\alpha} $ for all $\alpha \in \Gamma$.
This can be seen by representing $J$ in matrix form relative to the ordered basis of~$V$.
Suppose that $ \alpha \neq -\alpha $ and $ \alpha' = -\alpha'$ for $ \alpha, \alpha' \in \Gamma,$ and they are ordered as $ \cdots < \alpha < \cdots < \alpha' < \cdots < - \alpha < \cdots $.
Then $J$ has the following block matrix structure:
\begin{equation*}
	J =
	\bordermatrix{ & \alpha & \cdots & \alpha' & \cdots& -\alpha & \cdots \cr
		\alpha & \bm{\cdot} & \cdots & \bm{\cdot} & \cdots & \ast & \cdots \cr
		\; \vdots & \vdots & & \vdots & & \vdots & \cdots \cr
		\alpha' & \bm{\cdot} & \cdots & \ast & \cdots & \bm{\cdot} & \cdots \cr
		\; \vdots & \vdots & & \vdots & & \vdots & \cdots \cr
		-\alpha & \ast & \cdots & \bm{\cdot} & \cdots & \bm{\cdot} & \cdots \cr
		\;\vdots & \vdots & & \vdots & & \vdots &
	},
\end{equation*}
where the symbols $ \bm{\cdot}$ and $\ast$ have the same meaning as in \eqref{Edef}.
Note that the rows and columns intersecting a $\ast$ block contain only zero entries.
If $\dim V_{\alpha} \neq \dim V_{-\alpha},$ then the $(\alpha,-\alpha)$ and $(-\alpha,\alpha)$ blocks are rectangular matrices (not square).
This implies that $\det J = 0$, since one necessarily selects a zero entry when choosing one entry from each row and each column.
This shows that $J$ is degenerate.

We are now in a position to define subalgebras of $\gl(V,\omega)$ using the bilinear form $J.$

\begin{Proposition}[\cite{scheuTen}]\label{Prop:SulAlg}
	Let $ u$, $v$ be homogeneous elements of $V$ and $ J$ be a degree-zero nondegenerate bilinear form on $V.$
	Suppose that $J$ is either symmetric or skew-symmetric.
	Then the set of all homogeneous elements of $\gl(V,\omega)$ satisfying
\begin{equation}
J(Xu,v)+ \omega(\gamma,\alpha) J(u,Xv) = 0, \qquad X \in \gl(V,\omega)_{\gamma},\quad \forall u \in V_{\alpha}, \quad v \in V_{\beta}	\label{SubalgCond}
\end{equation}
	forms a $\Gamma$-graded subalgebra of $\gl(V,\omega)$.
\end{Proposition}

\begin{proof}
	It is easy to verify that if $ X \in \gl(V,\omega)_{\gamma} $ and $ Y \in \gl(V,\omega)_{\delta}$ satisfy \eqref{SubalgCond}, then $ \llbracket X, Y \rrbracket \in \gl(V,\omega)_{\gamma+\delta}$ also satisfies~\eqref{SubalgCond}.
	
	Next, consider the relation obtained by exchanging $ u $ and $ v $ in \eqref{SubalgCond}:
	\begin{equation}
		J(Xv,u) + \omega(\gamma,\beta) J(v, Xu) = 0. \label{SubalgCond2}
	\end{equation}
	A sufficient condition for the compatibility of this relation with \eqref{SubalgCond} is
	\begin{equation}
		J(u,v) = \pm \omega(\alpha,\beta) J(v,u). \label{Jsymm}
	\end{equation}
	Using this relation and the properties of the commutative factor given in \eqref{omegaprop1} and \eqref{omegaprop2}, one can see that the left hand side of \eqref{SubalgCond2} reduces to that of \eqref{SubalgCond}.
	Recall that $ J$ is nonvanishing if and only if $ \alpha + \beta = 0.$
	In this case, \eqref{Jsymm} yields
	\begin{equation*}
		J(u,v) = \pm \omega(\alpha,-\alpha) J(v,u) \stackrel{\eqref{omegaprop2}}{=} \pm J(v,u).
	\end{equation*}
Hence, when the bilinear form $J$ is symmetric or skew-symmetric, it is compatible with~\eqref{SubalgCond}.
\end{proof}

One may define color extensions of orthogonal, symplectic, and orthosymplectic Lie superalgebras, and so on.
A study in this direction, as well as the quantization of color Lie algebras and their affine analogue, is carried out in~\cite{zhang2026}.

Finally, we define a graded representation.

\begin{Definition}
	A graded representation $(\rho,V) $ of a color Lie algebra $\g$ in a $\Gamma$-graded vector space $V$ is a color Lie algebra homomorphism $ \rho\colon \g \to \gl(V,\omega)$ of degree zero:
	\begin{equation*}
		\rho(\llbracket X, Y \rrbracket) = \llbracket \rho(X), \rho(Y) \rrbracket,
		\qquad X, Y \in \g.
	\end{equation*}
\end{Definition}

\section{Graded Casimir elements and graded central extensions} \label{SEC:Method}

\subsection[Invariant bilinear form on color Lie algebras and 2nd order Casimir elements]{Invariant bilinear form on color Lie algebras\\ and 2nd order Casimir elements}

Let $\g$ be a color Lie algebra defined for a commutative factor $\omega$ on an Abelian group $\Gamma$.
We begin with some definitions.

\begin{Definition}\label{DEF:gcenter}
	The graded center $Z(\g)$ of $\g$ is the set of elements which graded-commute with all elements of $\g$:
	\begin{equation*}
		Z(\g) = \{ X \in \g \mid \llbracket X, Y \rrbracket = 0, \ \forall Y \in \g \}.
	\end{equation*}
\end{Definition}

Let $T(\g) := \sum_{r=0}^{\infty	} \g^{\otimes r} $ be the tensor algebra over~$\g$.
$T(\g)$ has $ \mathbb{Z}_+ \times \Gamma$-grading. The $\mathbb{Z}_+$ grading is obvious and the $\Gamma$-grading is given by
\begin{equation*}
	T(\g) = \sum_{\alpha\in \Gamma} T(\g)_{\alpha},
	\qquad
	T(\g)_{\alpha} = \sum_{r=0}^{\infty}\biggl( \sum_{\alpha_1+ \cdots + \alpha_r = \alpha} \g_{\alpha_1} \otimes \g_{\alpha_2} \otimes \cdots \otimes \g_{\alpha_r} \biggr).
\end{equation*}
Let $I(\g)$ be the two-sided ideal of $T(\g)$ generated by
\begin{equation*}
	X \otimes Y - \omega(\alpha,\beta) Y \otimes X - \llbracket X, Y \rrbracket, \qquad X \in \g_{\alpha},\quad Y \in \g_{\beta}.
\end{equation*}
The universal enveloping algebra of $\g$ is defined by $U(\g) = T(\g)/I(\g).$
Obviously, $U(\g)$ is a unital associative $\Gamma$-graded algebra.

\begin{Definition}
	Graded Casimir elements of $\g$ are elements of the graded center of $U(\g)$.
	Namely, if $C \in U(\g)$ is a Casimir element of degree $\gamma$, $ C X = \omega(\gamma, \alpha) XC $ for all homogeneous elements $ X \in U(\g) $ of degree~$\alpha$.
\end{Definition}

We denote an ordered basis of $\g_{\alpha}$ by $ X_{\alpha_i}$, $1 \leq i \leq \dim \g_{\alpha}$, and write the defining relations of~$\g$~as
\begin{align}
	& 	\llbracket X_{\alpha_i}, X_{\beta_j} \rrbracket = \sum_{\gamma_k} f_{\alpha_i, \beta_j}^{\quad\ \gamma_k} X_{\gamma_k}, \qquad f_{\alpha_i, \beta_j}^{\quad\ \gamma_k} \in \mathbb{C}, \label{sumgammak}
	\\
	& f_{\alpha_i, \beta_j}^{\quad\ \gamma_k} = 0 \qquad \text{if} \quad \alpha + \beta \neq \gamma, \nn
\end{align}
where $\sum_{\gamma_k}$ is a shorthand notation for $\sum_{\gamma \in \Gamma} \sum_{k=1}^{\dim \g_{\gamma}}$.
The relation
\[
\llbracket X_{\alpha_i}, X_{\beta_j} \rrbracket = -\omega(\alpha,\beta) \llbracket X_{\beta_j}, X_{\alpha_i} \rrbracket
\]
 is equivalent to the relation
\begin{equation}
	f_{\alpha_i, \beta_j}^{\quad\ \gamma_k} = -\omega(\alpha,\beta) f_{\beta_j,\alpha_i}^{\quad\ \gamma_k}.
	\label{fskew}
\end{equation}
Let $(\rho,V)$ be an irreducible graded representation of $\g$ in a $\Gamma$-graded vector space $V$.
If a~degree~$\mu$ element $ M^{\mu} \in \operatorname{End}(V)$ graded-commutes with $\rho(X)$ for all $ X \in \g$, i.e.,
\begin{equation*}
	\llbracket M^{\mu}, \rho(X) \rrbracket = 0, \qquad \forall X \in \g,
\end{equation*}
then we call $M^{\mu}$ a \textit{commutant} of degree $\mu$.
Clearly, the degree zero commutant is (a multiple of) the identity.

\begin{Proposition} \label{PROP:BF}
Assume $M^{-\mu}$ is a commutant of degree $-\mu$.
	We define a bilinear form $\eta^{\mu}\colon \g \times \g \to \mathbb{C}$ by
	\begin{equation*}
		\eta^{\mu}(X_{\alpha_i}, X_{\beta_j}) = \operatorname{ctr}( \rho(X_{\alpha_i}) M^{-\mu} \rho(X_{\beta_j}) ), 
	\end{equation*}
	which is also denoted by $\eta^{\mu}_{\alpha_i,\beta_j}$. This bilinear form satisfies the following properties:
	\begin{enumerate}\itemsep=0pt
		\item[$(1)$] $ \eta^{\mu}_{\alpha_i, \beta_j} = 0$ if $ \alpha + \beta \neq \mu$,
		\item[$(2)$] $ \eta^{\mu}_{\beta_j,\alpha_i} = \omega(\mu,\mu) \omega(\alpha,\beta) \eta^{\mu}_{\alpha_i, \beta_j} $,
		\item[$(3)$] $ \eta^{\mu}(\llbracket X_{\alpha_i}, X_{\beta_j} \rrbracket, X_{\gamma_k} ) = \omega(\mu,\beta) \eta^{\mu}(X_{\alpha_i}, \llbracket X_{\beta_j}, X_{\gamma_k} \rrbracket) $.
	\end{enumerate}
	Relation $(3)$ implies the $\g$-invariance of $\eta^{\mu}$.
\end{Proposition}

\begin{proof}
	To simplify the notations, we set $ A := \rho(X_{\alpha_i})$, $B := \rho(X_{\beta_j})$, $C := \rho(X_{\gamma_k})$.

(1) By the definition of the color trace, $\operatorname{ctr} X = 0 $ if $\deg X \neq 0.$
		The statement immediately follows from $ \deg (AM^{-\mu}B) = \alpha + \beta - \mu$.

(2) The quantity $\eta^{\mu}_{\beta_j,\alpha_i} = \operatorname{ctr}(B M^{-\mu} A) $ is computed, using \eqref{ctrperm} and $ \llbracket M^{-\mu}, A \rrbracket = 0, $ as follows:
		\begin{align*}
			\eta^{\mu}_{\beta_j,\alpha_i} &
			\stackrel{\eqref{ctrperm}}{=} \omega(\beta,-\mu+\alpha)\, \operatorname{ctr}(M^{-\mu} A B)
			= \omega(\beta,-\mu+\alpha) \omega(-\mu,\alpha)\, \operatorname{ctr}(A M^{-\mu} B)
			\nn \\
			&\, \, = \omega(\beta,-\mu+\alpha) \omega(-\mu,\alpha)\, \eta^{\mu}_{\alpha_i, \beta_j}.
		\end{align*}
		Recalling that $ \eta^{\mu}_{\alpha_i, \beta_j} \neq 0 $ only if $ \alpha+\beta = \mu,$
		it is straightforward to verify the identity
\[
\omega(\beta,-\mu+\alpha) \omega(-\mu,\alpha) = \omega(\mu,\mu) \omega(\alpha,\beta).
\]

(3) By definition, we can write
		\begin{align*}
			\eta^{\mu}(\llbracket X_{\alpha_i}, X_{\beta_j} \rrbracket, X_{\gamma_k} ) =
			\operatorname{ctr}(\llbracket A, B \rrbracket M^{-\mu}C)
			=\operatorname{ctr}\big( (AB - \omega(\alpha,\beta) BA) M^{-\mu}C \big).
		\end{align*}
		Using \eqref{ctrperm} and \eqref{omegaprop2}, we obtain
		\begin{align*}
			\eta^{\mu}(\llbracket X_{\alpha_i}, X_{\beta_j} \rrbracket, X_{\gamma_k} )
			&= \operatorname{ctr}\big( \omega(\beta,-\mu) A M^{-\mu} BC - \omega(\alpha,\beta) \omega(\beta,\alpha-\mu+\gamma) A M^{-\mu}CB \big)
			\nn \\
			&= \omega(\mu,\beta) \operatorname{ctr}\big( A M^{-\mu} (BC - \omega(\beta,\gamma) CB)) \big)
			\nn \\
			&= \omega(\mu,\beta) \operatorname{ctr} (A M^{-\mu} \llbracket B, C \rrbracket) = \omega(\mu,\beta) \eta^{\mu}(X_{\alpha_i}, \llbracket X_{\beta_j}, X_{\gamma_k} \rrbracket).
		\end{align*}
This completes the proof.
\end{proof}

In terms of the structure constants $f_{\alpha_i, \beta_j}^{\quad\ \gamma_k}$, the $\g$-invariance of $\eta^{\mu}$ takes the following form:
\begin{equation}
	\sum_{\nu_{\ell}}
	f_{\alpha_i \beta_j}^{\quad \ \nu_{\ell}} \eta^{\mu}_{\nu_{\ell}\gamma_k}
	=
	\omega(\mu,\beta)
	\sum_{\nu_{\ell}}
	f_{\beta_j \gamma_k}^{\quad \ \nu_{\ell}} \eta^{\mu}_{\alpha_i \nu_{\ell}},
	\label{Invf}
\end{equation}
where we continue to use the shorthand notation introduced in~\eqref{sumgammak}.
If $\eta^{\mu}$ is nondegenerate, we introduce its inverse and denote it by $ \eta_{\mu}$:
\begin{align*}
	 &\eta_{\mu}^{\alpha_i \beta_j}=\eta_{\mu}(X_{\alpha_i}, X_{\beta_j}) = 0,\qquad \alpha + \beta \neq \mu,
	\qquad
	 \sum_{\gamma_k } \eta_{\mu}^{\alpha_i \gamma_k} \eta^{\mu}_{\gamma_k \beta_j}=\sum_{\gamma_k}
	\eta^{\mu}_{\beta_j \gamma_k} \eta_{\mu}^{\gamma_k \alpha_i} = \delta^{\alpha_i}_{\beta_j}.
\end{align*}
An equivalent expression to \eqref{Invf} in terms of $\eta_{\mu}$ is given by
\begin{equation}
	\sum_{\nu_{\ell}}
	\eta_{\mu}^{\alpha_i \nu_{\ell}} f_{\nu_{\ell} \beta_j}^{\quad \ \gamma_k}
	= \omega(\mu,\beta)
	\sum_{\nu_{\ell}}
	\eta_{\mu}^{\nu_{\ell} \gamma_k} f_{\beta_j \nu_{\ell}}^{\quad\ \alpha_i}.
	\label{Invf2}
\end{equation}

\begin{Theorem}
	A second order Casimir element of degree $\mu$ is given by
	\begin{equation*}
		C_{\mu} = \sum_{\alpha_i,\beta_j} \eta_{\mu}^{\alpha_i \beta_j} X_{\alpha_i} X_{\beta_j}.
	\end{equation*}
\end{Theorem}

\begin{proof}
	It is sufficient to prove $ \llbracket X_{\alpha_i}, C_{\mu} \rrbracket = 0. $
	It is proved by straightforward computation:
	\begin{align*}
		\llbracket X_{\alpha_i}, C_{\mu} \rrbracket &= \llbracket X_{\alpha_i}, \sum_{\beta_j,\gamma_k} \eta_{\mu}^{\beta_j \gamma_k} X_{\beta_j} X_{\gamma_k} \rrbracket
		\nn\\
		&= \sum_{\beta_j,\gamma_k,\nu_{\ell}} \eta_{\mu}^{\beta_j \gamma_k}
		\big( f_{\alpha_i \beta_j}^{\quad \ \nu_{\ell}} X_{\nu_{\ell}} X_{\gamma_k} + \omega(\alpha,\beta) f_{\alpha_i \gamma_k}^{\quad\ \nu_{\ell}} X_{\beta_j} X_{\nu_{\ell}} \big)
		\nn\\
		&=
		\sum_{\beta_j,\gamma_k,\nu_{\ell}}
		\big( \eta_{\mu}^{\beta_j \gamma_k} f_{\alpha_i \beta_j}^{\quad \ \nu_{\ell}} + \omega(\alpha,\nu) \eta_{\mu}^{\nu_{\ell} \beta_j} f_{\alpha_i \beta_j}^{\quad \ \gamma_k} \big)
		X_{\nu_{\ell}} X_{\gamma_k}
		\nn \\
		&\stackrel{\eqref{fskew}}{=}
		\sum_{\beta_j,\gamma_k,\nu_{\ell}}
		\big( \eta_{\mu}^{\beta_j \gamma_k} f_{\alpha_i \beta_j}^{\quad \ \nu_{\ell}} - \omega(\alpha,\nu) \omega(\alpha,\beta)\, \eta_{\mu}^{\nu_{\ell} \beta_j} f_{\beta_j \alpha_i}^{\quad \ \gamma_k} \big)
		X_{\nu_{\ell}} X_{\gamma_k}
		\nn\\
		&\stackrel{\eqref{Invf2}}{=}
		\sum_{\beta_j,\gamma_k,\nu_{\ell}} \eta_{\mu}^{\nu_{\ell} \beta_j} f_{\beta_j \alpha_i}^{\quad \ \gamma_k} \big( \omega(\alpha,\mu) - \omega(\alpha,\nu+\beta) \big)
		X_{\nu_{\ell}} X_{\gamma_k}
		 \stackrel{\nu+\beta=\mu}{=} 0.\tag*{\qed}
	\end{align*}\renewcommand{\qed}{}
\end{proof}

The essence of this construction of graded Casimir elements lies in a commutant of nontrivial degree.
Thus, the construction depends on the representation $\rho$, but it does not depend on the choice of basis of $\g$.
Such commutants have been found for low-dimensional color Lie algebras with $\Gamma = \Z2$ \cite{AFITTsuper,AIKTTslint,AiSe}.

\subsection{Loop extension of color Lie algebras and graded central extensions}

In this subsection, we discuss graded central extensions of the loop color algebra (see Definition~\ref{DEF:gcenter} for the graded center).
The loop algebra associated with a color Lie algebra $\g$ is defined as $ L(\g) = \g \otimes \mathbb{C}[\lambda,\lambda^{-1}]$, where $\lambda$ is a spectral parameter of degree zero.
Elements of $L(\g)$ have the form~$ X \otimes \lambda^n$, where $ X \in \g$ and $ n \in \mathbb{Z}$. The color Lie bracket is defined by
\begin{equation*}
	\llbracket X\otimes \lambda^m, Y \otimes \lambda^n \rrbracket = \llbracket X, Y \rrbracket \otimes \lambda^{m+n}.
\end{equation*}
The algebra $L(\g)$ inherits a natural $\Gamma$-grading from that of $\g$.

For the ordered basis $ X_{\alpha_i} $ of $\g$, we set $ X_{\alpha_i}^{(m)} := X_{\alpha_i} \otimes \lambda^{m}$. Then the set $ \bigr\{ X_{\alpha_i}^{(m)}\bigl\}$ forms a~basis of $L(\g), $
and the defining relations of $L(\g)$ take the form of
\begin{equation*}
	\bigl\llbracket X_{\alpha_i}^{(m)}, X_{\beta_j}^{(n)} \bigr\rrbracket = \sum_{\gamma_k} f_{\alpha_i \beta_j}^{\quad\ \gamma_k} X_{\gamma_k}^{(m+n)},
	\qquad m, n \in \mathbb{Z}.
\end{equation*}

\begin{Theorem} \label{THM:centralext}
	Following the notation of the previous subsection, a graded central extension of $L(\g)$ is given by
	\begin{align}
		\bigl\llbracket X_{\alpha_i}^{(m)}, X_{\beta_j}^{(n)}\bigr \rrbracket = \sum_{\gamma_k} f_{\alpha_i \beta_j}^{\quad\ \gamma_k} X_{\gamma_k}^{(m+n)}
		+ \sum_{\mu\in \Gamma} m \delta_{n+m,0} \omega(\alpha,\mu) \eta^{\mu}_{\alpha_i \beta_j} c_{\mu},
		\label{comWc}
	\end{align}
	where $ c_{\mu}$ is a graded center of degree $\mu$:
	\begin{equation*}
		\bigl\llbracket c_{\mu}, X_{\alpha_i}^{(m)} \bigr\rrbracket = \llbracket c_{\mu}, c_{\nu} \rrbracket = 0, \qquad \forall X_{\alpha}^{(m)} \in L(\g), \quad \forall \mu, \nu \in \Gamma.
	\end{equation*}
\end{Theorem}

\begin{proof}
	We prove that the relation \eqref{comWc} satisfies the color Jacobi identity \eqref{colorJacobi}.
	\begin{align*}
		\llbracket X_{\alpha_i}^{(m)}, \llbracket X_{\beta_j}^{(n)} X_{\gamma_k}^{(\ell)} \rrbracket \rrbracket
		 = \sum_{\nu_p} f_{\beta_j,\gamma_k}^{\quad \ \nu_p}
		\biggl( \sum_{\sigma_q} f_{\alpha_i \nu_p}^{\quad\ \sigma_q} X_{\sigma_q}^{(m+n+\ell)} + \sum_{\mu \in \Gamma} m \delta_{m+n+\ell,0} \omega(\alpha,\mu) \eta_{\alpha_i,\nu_p}^{\mu} c_{\mu} \biggr).
	\end{align*}
	The first term has the same structure as that of $L(\g)$, hence it satisfies the color Jacobi identity.
	Therefore, it suffices to keep only the graded central terms;
	we shall denote the central term of an expression by $(\ldots)_{\mathrm{ct}}$.
	By the $\g$-invariance of $\eta^{\mu} $ \eqref{Invf} and the fact that the structure constants are nonvanishing only when $ \alpha + \beta = \nu, $ this takes the form
	\begin{equation*}
		\bigl(\bigl\llbracket X_{\alpha_i}^{(m)}, \bigl\llbracket X_{\beta_j}^{(n)} X_{\gamma_k}^{(\ell)} \bigr\rrbracket \bigr\rrbracket \bigr)_{\mathrm{ct}}
		= m \delta_{m+n+\ell,0} \sum_{\nu_p,\mu} \omega(\nu,\mu) f_{\alpha_i \beta_j}^{\quad \ \nu_p} \eta_{\nu_p \gamma_k}^{\mu} c_{\mu}.
	\end{equation*} 	
	Similarly, we obtain the following expressions (keeping only the graded central terms):
	\begin{align*}
\bigl(\bigl\llbracket \bigl\llbracket X_{\alpha_i}^{(m)}, X_{\beta_j}^{(n)} \bigr\rrbracket, X_{\gamma_k}^{(\ell)} \bigr\rrbracket \bigr)_{\mathrm{ct}}
		&= (m+n) \delta_{m+n+\ell,0} \sum_{\nu_p,\mu} \omega(\nu,\mu) f_{\alpha_i \beta_j}^{\quad \ \nu_p} \eta_{\nu_p \gamma_k}^{\mu} c_{\mu},
	\end{align*}
	and
	\begin{align*}
\bigl(\bigl\llbracket X_{\beta_j}^{(n)}, \bigl\llbracket X_{\alpha_i}^{(m)}, X_{\gamma_k}^{(\ell)} \bigr\rrbracket \bigr\rrbracket\bigr)_{\mathrm{ct}}
		&= n \delta_{m+n+\ell,0}\sum_{\nu_p,\mu} \omega(\beta,\mu)f_{\alpha_i \gamma_k}^{\quad \ \nu_p} \eta_{\beta_j \nu_p}^{\mu} c_{\mu}
		\nn\\
		&\stackrel{\eqref{Invf}}{=} n \delta_{m+n+\ell,0}\sum_{\nu_p,\mu} \omega(\beta,\mu) \omega(\alpha,\mu) f_{\beta_j \alpha_i }^{\quad \ \nu_p} \eta_{\nu_p \gamma_k}^{\mu} c_{\mu}
		\nn\\
		&\stackrel{\eqref{fskew}}{=} -n \delta_{m+n+\ell,0}\sum_{\nu_p,\mu} \omega(\alpha+\beta,\mu) \omega(\beta,\alpha) f_{\alpha_i \beta_j }^{\quad \ \nu_p} \eta_{\nu_p \gamma_k}^{\mu} c_{\mu}.
	\end{align*}
	These relations show that the color Jacobi identity is satisfied.
\end{proof}

The graded central extension has been discussed for particular $\Gamma = \Z2$-graded color loop algebras $L(\sl(2))$ and $L(\osp(1|2))$ and the existence of a graded central element of degree ${(1,1) \in \Z2}$ has been observed in \cite{AIKTTslint,AiSe}.
A $\Z2$-graded extension of the Virasoro algebra was also discussed in these works in relation to integrable $\Z2$-graded extensions of the Liouville equation and the Sugawara construction. Furthermore, a $\Z2$-graded extension of the super-Virasoro algebra was studied in connection with an integrable $\Z2$-super-Liouville equation~\cite{AFITTsuper}.

\section{Color algebras having nontrivial commutant} \label{SEC:Eg1}

As mentioned in the previous section, the commutant of degree $0$ is trivial, i.e., it is given by an identity matrix.
In this and the next sections, we present color Lie algebras which have commutants of nontrivial degree and explicit formulas of their 2nd order Casimir elements.

\subsection[Z\_2\^{}2-graded extension of q(n)]{$\boldsymbol{\Z2}$-graded extension of $\boldsymbol{\q(n)}$} \label{SEC:qn}

In this subsection, we introduce a $\Gamma = \Z2$-graded extension of the strange Lie superalgebra $\q(n)$ and its affine extension.

Let us recall that the Lie superalgebra $\q(n)$ can be determined through its defining representation, i.e.\
\begin{equation*}
	\q(n) = \left\{
	\left( \begin{matrix} A&B\\B&A \end{matrix} \right) \mid
	A,B\in \gl(n) \right\},
\end{equation*}
where the matrices with $B=0$ are even, or elements of $\q(n)_{0}$, and those
with $A=0$ odd, or elements of $\q(n)_{1}$.
With $e_{ij}$ the notation for the $n\times n$ matrix unit,
a basis for $\q(n)$ is given by%
\begin{align}
	&E_{ij}^0 = \left( \begin{matrix} 1&0\\0&1 \end{matrix} \right)\otimes e_{ij},\qquad
	E_{ij}^1 = \left( \begin{matrix} 0&1\\1&0 \end{matrix} \right)\otimes e_{ij},\qquad
	i,j\in\{1,\ldots,n\},
	\nn\\
	&E_{ij}^{\alpha} \in \q(n)_{\alpha}, \qquad \alpha \in \mathbb{Z}_2.
	\label{qnbasis}
\end{align}
Note that the $ 2\times 2$ matrices in \eqref{qnbasis} form a $\mathbb{Z}_2$-graded algebra.

We extend the representation \eqref{qnbasis} to $\Z2$-graded setting.
Let $\Gamma=\Z2=\{00,01,10,11\}$ and
$\omega(\alpha,\beta)=(-1)^{\alpha_1\beta_1+\alpha_2\beta_2}$ for $\alpha = \alpha_1\alpha_2$, $\beta =\beta_1\beta_2 \in\Gamma$.
We introduce a $\Z2$-graded color algebra defined by the matrices
\begin{align}
	&
	\Theta^{00}= \left(\begin{matrix}
		1 & 0 & 0 & 0
		\\
		0 & 1 & 0 & 0
		\\
		0 & 0 & 1 & 0
		\\
		0 & 0 & 0 & 1
	\end{matrix}\right), \qquad
	\Theta^{10}=\left(\begin{matrix}
		0 & 0 & 1 & 0
		\\
		0 & 0 & 0 & 1
		\\
		1 & 0 & 0 & 0
		\\
		0 & 1 & 0 & 0
	\end{matrix}\right) \nonumber\\
	&
	\Theta^{01}=\left(\begin{matrix}
		0 & 1 & 0 & 0
		\\
		1 & 0 & 0 & 0
		\\
		0 & 0 & 0 & 1
		\\
		0 & 0 & 1 & 0
	\end{matrix}\right),\qquad
	\Theta^{11}=\left(\begin{matrix}
		0 & 0 & 0 & 1
		\\
		0 & 0 & 1 & 0
		\\
		0 & 1 & 0 & 0
		\\
		1 & 0 & 0 & 0
	\end{matrix}\right),
	\label{Theta}
\end{align}
which satisfy the product relations
\begin{equation}
	\Theta^\alpha\Theta^\beta=\Theta^{\alpha+\beta}.
	\label{OO}
\end{equation}
Consider now the $4n^2$ homogeneous elements
\[
E_{ij}^\alpha = \Theta^\alpha \otimes e_{ij}, \qquad i,j\in\{1,\ldots,n\},\quad \alpha\in\Z2,
\]
which act on the $\Z2$-graded vector space
$ \mathbb{C}^{n,n|n,n} := \mathbb{C}^n \oplus \mathbb{C}^n \oplus \mathbb{C}^n \oplus \mathbb{C}^n.$ These elements form a basis of the $\Z2$-graded extension of $\q(n)$, denoted by \qq, whose
bracket relation is
\begin{equation*}
	\bigl\llbracket E_{ij}^\alpha, E_{kl}^\beta \bigr\rrbracket =
	\delta_{jk} E_{il}^{\alpha+\beta} - \omega(\alpha,\beta) \delta_{il}E_{kj}^{\alpha+\beta}.
\end{equation*}
The color Lie algebra \qq \ can be considered as a subalgebra of the color Lie algebra $\gl(n,n|n,n) :=\gl(V,\omega)$
with $\Gamma=\Z2$, $V = \mathbb{C}^{n,n|n,n}$ and $\omega(\alpha,\beta)=(-1)^{\alpha_1\beta_1+\alpha_2\beta_2}$

In order to study commutants, let us define the following matrices:
\begin{alignat*}{3}
	&
	T^{00}= \left(\begin{matrix}
		1 & 0 & 0 & 0
		\\
		0 & 1 & 0 & 0
		\\
		0 & 0 & 1 & 0
		\\
		0 & 0 & 0 & 1
	\end{matrix}\right), \qquad&&
	T^{10}=\left(\begin{matrix}
		0 & 0 & 1 & 0
		\\
		0 & 0 & 0 & 1
		\\
		-1 & 0 & 0 & 0
		\\
		0 & -1 & 0 & 0
	\end{matrix}\right), & \nonumber\\
	&
	T^{01}=\left(\begin{matrix}
		0 & 1 & 0 & 0
		\\
		- 1 & 0 & 0 & 0
		\\
		0 & 0 & 0 & 1
		\\
		0 & 0 & -1 & 0
	\end{matrix}\right),\qquad&&
	T^{11}=\left(\begin{matrix}
		0 & 0 & 0 & 1
		\\
		0 & 0 & -1 & 0
		\\
		0 & -1 & 0 & 0
		\\
		1 & 0 & 0 & 0
	\end{matrix}\right).&
\end{alignat*}
Then one can verify that
\begin{equation}
	\Theta^\alpha T^\beta = \omega(\alpha,\beta) T^\beta \Theta^\alpha.
	\label{OT}
\end{equation}
Now define the following elements of $\gl(n,n|n,n)$:
\begin{equation*}
	M^\alpha = T^\alpha\otimes I_n.
\end{equation*}
It is easy to prove that the matrices $M^\alpha$ are commutants in the defining representation of \qq.
Indeed,
\begin{align*}
	\bigl\llbracket M^\alpha, E^\beta_{ij}\bigr\rrbracket & = (T^\alpha\otimes I_n)\bigl(\Theta^\beta\otimes E_{ij}\bigr)
	- \omega(\alpha,\beta) \bigl(\Theta^\beta\otimes E_{ij}\bigr)(T^\alpha\otimes I_n) \nonumber\\
	&= \bigl(T^\alpha\Theta^\beta - \omega(\alpha,\beta) \Theta^\beta T^\alpha\bigr) \otimes E_{ij} =0,
\end{align*}
using~\eqref{OT}.
So we have 4 commutants, one of each degree in $\Gamma$.

Next, we compute the corresponding bilinear forms on \qq, using
\begin{equation*}
	\eta^{-\mu} \bigl(E_{ij}^\alpha, E_{kl}^\beta\bigr)=\operatorname{ctr}\bigl(E_{ij}^\alpha M^\mu E_{kl}^\beta\bigr).
\end{equation*}
For $\mu=00$, one finds
\begin{align*}
	E_{ij}^\alpha M^{00} E_{kl}^\beta =
	E_{ij}^\alpha E_{kl}^\beta=
	\bigl(\Theta^\alpha \Theta^\beta\bigr) \otimes (E_{ij}E_{kl}) =
	  \Theta^{\alpha+\beta} \otimes \delta_{jk}E_{il} = \delta_{jk} E_{il}^{\alpha+\beta}.
\end{align*}
Note that the color trace is equivalent to the supertrace for $\Gamma = \Z2.$
Hence,
\begin{equation*}
	\eta^{00} \bigl(E_{ij}^\alpha, E_{kl}^\beta\bigr)=\operatorname{ctr}\bigl(\delta_{jk} E_{il}^{\alpha+\beta}\bigr) =
	\delta_{jk} \operatorname{ctr}\bigl( \Theta^{\alpha+\beta} \otimes E_{ij} \bigr)= \delta_{jk} \operatorname{ctr}\bigl(\Theta^{\alpha+\beta}\bigr)\operatorname{tr}(E_{ij}) = 0,
\end{equation*}
since for the $4\times 4$ matrices $A$ of type~\eqref{Theta} one has $\operatorname{ctr}(A)=a_{11}-a_{22}-a_{33}+a_{44}$.
This implies that $\eta^{00}=0$, so it is degenerate, and one cannot construct a Casimir of order $00$.
This situation is analogous to that of the Lie superalgebra $\q(n)$, which is known not to admit a second-order Casimir~\cite{Sergeev1983}.

For $\mu=11$, one finds
\begin{equation*}
	E_{ij}^\alpha M^{11} E_{kl}^\beta=
	\bigl(\Theta^\alpha T^{11} \Theta^\beta\bigr) \otimes (E_{ij}E_{kl}).
\end{equation*}
Then
\begin{equation*}
	\operatorname{ctr}\bigl(E_{ij}^\alpha M^{11} E_{kl}^\beta\bigr) = \operatorname{ctr}\bigl(\Theta^\alpha T^{11} \Theta^\beta\bigr) \operatorname{tr}(E_{ij}E_{kl} )
	= \delta_{jk}\delta_{il} \operatorname{ctr}\bigl(\Theta^\alpha T^{11} \Theta^\beta\bigr).
\end{equation*}
Using~\eqref{OT} and \eqref{OO}, one finds
\begin{equation*}
	\operatorname{ctr}\bigl(\Theta^\alpha T^{11} \Theta^\beta\bigr)= \omega(\alpha,11) \operatorname{ctr}\bigl(T^{11} \Theta^\alpha \Theta^\beta\bigr)=
	\omega(\alpha,11) \operatorname{ctr}\bigl(T^{11} \Theta^{\alpha+\beta}\bigr).
\end{equation*}
From the explicit form of $T^{11}$ and the matrices $\Theta^\gamma$, it follows that
$\operatorname{ctr}\bigl(T^{11} \Theta^{\gamma}\bigr)$ is nonzero only for $\gamma=11$, and in that case
$\operatorname{ctr}\bigl(T^{11} \Theta^{11}\bigr)=4$.
Thus we have
\begin{equation*}
	\operatorname{ctr}\bigl(\Theta^\alpha T^{11} \Theta^\beta\bigr)= 4\omega(\alpha,11)\delta_{\alpha+\beta,11},
\end{equation*}
and
\begin{equation*}
	\eta^{11}\bigl(E_{ij}^\alpha, E_{kl}^\beta\bigr)=\operatorname{ctr}\bigl(E_{ij}^\alpha M^{11} E_{kl}^\beta\bigr) = 4 \delta_{jk}\delta_{il}\omega(\alpha,11)\delta_{\alpha+\beta,11}.
\end{equation*}
The factor $4$ does not play a role, so for convenience we redefine
\begin{equation*}
	\eta^{11} \bigl(E_{ij}^\alpha, E_{kl}^\beta\bigr):= \delta_{jk}\delta_{il}\omega(\alpha,11)\delta_{\alpha+\beta,11}.
\end{equation*}
This bilinear form is nondegenerate, and its inverse reads
\begin{equation*}
	\eta_{11} \bigl(E_{ij}^\alpha, E_{kl}^\beta\bigr)= \delta_{jk}\delta_{il}\omega(\alpha,11)\delta_{\alpha+\beta,11}.
\end{equation*}
Thus we can construct a 2nd order Casimir element of degree~$11$:
\begin{align*}
	C_{11} = \sum_{\alpha,\beta}\sum_{ij}\sum_{kl}
	\delta_{jk}\delta_{il}\omega(\alpha,11)\delta_{\alpha+\beta,11} E_{ij}^\alpha E_{kl}^\beta
	 = \sum_{\alpha,i,j} \omega(\alpha,11) E_{ij}^\alpha E_{ji}^{11-\alpha}.
\end{align*}
Explicitly,
\begin{equation*}
	C_{11} = \sum_{ij} \bigl( E_{ij}^{00}E_{ji}^{11} - E_{ij}^{01}E_{ji}^{10} - E_{ij}^{10}E_{ji}^{01} + E_{ij}^{11}E_{ji}^{00} \bigr).
\end{equation*}

Finally, for $\mu=10$ and $\mu=01$, one finds that $\operatorname{ctr}\bigl(E_{ij}^\alpha M^\mu E_{kl}^\beta\bigr)=0$, hence $\eta^{\mu}=0$, and one cannot construct a Casimir of order $10$ or $01$.

Let us now turn to the loop algebra of \qq \ and denote its basis by $ E_{ij}^{\alpha (m)}$, where $ m \in \mathbb{Z}.$
According to Theorem \ref{THM:centralext}, the loop algebra admits a $11$-graded central extension given by the bilinear form $\eta^{11}$:
\begin{equation*}
	\bigl\llbracket E_{ij}^{\alpha(m)}, E_{kl}^{\beta(n)} \bigr\rrbracket =
	\delta_{jk} E_{il}^{\alpha+\beta(m+n)} - \omega(\alpha,\beta) \delta_{il}E_{kj}^{\alpha+\beta(m+n)}
	+ n \delta_{m+n,0} \delta_{\alpha+\beta,11} \delta_{jk} \delta_{il}c_{11}.
\end{equation*}
More explicitly, the relations containing $c_{11} $ are as follows:
\begin{equation*}
	\bigl\llbracket E_{ij}^{\alpha(m)}, E_{ji}^{11-\alpha(-m)} \bigr\rrbracket = E_{ii}^{11(0)} - E_{jj}^{11(0)} - m c_{11},
	\qquad \alpha+\beta = 11.
\end{equation*}

\subsection[Z\_3\^{}2-graded extension of sl(2)]{$\boldsymbol{\mathbb{Z}_3^2}$-graded extension of $\boldsymbol{\sl(2)}$} \label{SEC:sl2}

We consider a color Lie algebra extension of $\sl(2)$ with the grading group
\begin{equation}
	\Gamma = \mathbb{Z}_3^2 := \mathbb{Z}_3 \times \mathbb{Z}_3 =
	\{
	 00, 01, 02, 10, 11, 12, 20, 21, 22 	\}
\end{equation}
and the commutative factor
\begin{equation}
	\omega(\alpha,\beta) = \xi^{\alpha_1 \beta_2 - \alpha_2 \beta_1},
	\qquad
	\xi = e^{2\pi {\rm i} /3}, \qquad
	\alpha, \beta \in \Gamma .
	\label{Z32omega}
\end{equation}
For this setting, the color Lie bracket for some homogeneous elements is neither commutator nor anticommutator, see \eqref{ColorBracket}.

According to \cite{scheu}, we introduce a mapping $ \sigma\colon \Gamma \times \Gamma \to
\mathbb{C}_*$ such that
\begin{align*}
	&\omega(\alpha,\beta) = \sigma(\alpha,\beta) \sigma(\beta,\alpha)^{-1},
	\qquad
	\sigma(\alpha,\beta+\gamma) = \sigma(\alpha,\beta) \sigma(\alpha,\gamma),
	\\
	& \sigma(\alpha+\beta,\gamma) = \sigma(\alpha,\gamma) \sigma(\beta,\gamma).
\end{align*}
For the present $\Gamma$ and $ \omega $ given in \eqref{Z32omega}, the mapping $\sigma $ is given by
\begin{equation*}
	\sigma(\alpha,\beta) = \xi^{\alpha_1 \beta_2}.
\end{equation*}
With the mapping $\sigma$, we introduce a color algebra of dimension 9 whose relations are defined~by%
\begin{equation}
	e^{\alpha} e^{\beta} = \sigma(\alpha,\beta) e^{\alpha+\beta}, \qquad \alpha, \beta \in \Gamma.
	\label{GammaAlg}
\end{equation}
The $\mathbb{Z}_3^2$-graded extension of $\sl(2),$ denoted by \Zsl, is spanned by the following nine homogeneous elements:
\begin{alignat}{3}
	&H^{\alpha}= e^{\alpha} \otimes H, \qquad&& \alpha= 00, 11, 22,&
	\nn \\
	&E_+^{\alpha}= e^{\alpha} \otimes E_+, \qquad&& \alpha= 01, 12, 20,&
	\nn \\
	&E_-^{\alpha}= e^{\alpha} \otimes E_-, \qquad&& \alpha= 02, 21, 10,&
	\label{BasisZ3sl2}
\end{alignat}
where $ H$, $E_{\pm} $ is the basis of the Lie algebra $\sl(2)$ subject to the relations
\begin{equation}
	[H, E_{\pm}] = \pm 2 E_{\pm}, \qquad
	[E_+, E_-] = H. \label{Basisl2}
\end{equation}
The nonvanishing color Lie brackets for the basis \eqref{BasisZ3sl2} can be determined by~\eqref{GammaAlg} and~\eqref{Basisl2}~as
\begin{equation}
	\bigl\llbracket H^{\alpha}, E_{\pm}^{\beta} \bigr\rrbracket = \pm 2 \sigma(\alpha,\beta)E_{\pm}^{\alpha+\beta},
	\qquad
	\bigl\llbracket E_+^{\alpha}, E_-^{\beta} \bigr\rrbracket = \sigma(\alpha,\beta) H^{\alpha+\beta}.
	\label{Z32slrelations}
\end{equation}
\begin{Remark}
	These relations can also be expressed using the commutative factor $\omega(\alpha,\beta)$.
	We change the basis of \Zsl \ as
	\begin{align*}
		\tilde{H}^{22} = \xi H^{22}, \qquad \tilde{E}^{12}_+ = \xi E^{12}_+, \qquad
		\tilde{E}^{21}_- = \xi E^{21}_-,
	\end{align*}
	while keeping the other basis elements unchanged. Then the defining relations \eqref{Z32slrelations} become
	\begin{align*}
		& \bigl\llbracket \tilde{H}^{\alpha}, \tilde{E}_+^{\beta} \bigr\rrbracket = 2 \tilde{E}_+^{\alpha+\beta},
		\qquad
		\bigl\llbracket \tilde{H}^{\alpha}, \tilde{E}_-^{\beta} \bigr\rrbracket = -2\omega(\alpha,\beta) \tilde{E}_-^{\alpha+\beta},
		\qquad
		\bigl\llbracket \tilde{E}_+^{\alpha}, \tilde{E}_-^{\beta} \bigr\rrbracket = \tilde{H}^{\alpha+\beta}.
	\end{align*}
\end{Remark}

Next, to search for the commutants of \Zsl, we consider the representation $(\rho,V)$ of \Zsl \ on a $\mathbb{Z}_3^2$-graded vector space
\smash{$V = \bigoplus_{\alpha \in \mathbb{Z}_3^2} V_\alpha$}.
The basis $\ket{\alpha}\in V_{\alpha}$ is taken to be
\begin{alignat*}{4}
	&\ket{00}= e_{00} \otimes \ket{1}, \qquad&&
	\ket{11}= e_{11} \otimes \ket{1}, \qquad&&
	\ket{22}= e_{22} \otimes \ket{1},&
	\notag \\
	&\ket{02}= e_{02} \otimes \ket{-1}, \qquad &&
	\ket{21}= e_{21} \otimes \ket{-1}, \qquad &&
	\ket{10}= e_{10} \otimes \ket{-1},&
\end{alignat*}
where $\ket{\pm 1}$ is the basis of the fundamental representation of $\sl(2)$:
\begin{equation}
	H \ket{\pm 1} = \pm \ket{\pm 1}, \qquad E_{\pm} \ket{\pm 1} = 0,
	\qquad E_{\pm} \ket{\mp 1} = \ket{\pm 1}.
	\label{sl2Rep}
\end{equation}
The graded subspaces $ V_{12}$, $V_{01}$ and $ V_{20}$ are empty, while the remaining subspaces are one-dimensional.
The action of \eqref{BasisZ3sl2} on $V$ is computed as follows:
\begin{align*}
	&\rho(H^{00}) \ket{11}= (e_{00} \otimes H) (e_{11}\otimes \ket{1})
	= e_{00}e_{11} \otimes H \ket{1}
	\stackrel{\eqref{GammaAlg},\, \eqref{sl2Rep}}{=} e_{11} \otimes \ket{1}
	 = \ket{11}.
\end{align*}

Ordering the basis as
\begin{equation*}
	\ket{00}, \quad \ket{11}, \quad \ket{22}, \quad \ket{02}, \quad \ket{21}, \quad \ket{10},
\end{equation*}
the representation $(\rho,V)$ is given explicitly as follows.
First, $\rho\bigl(H^{00}\bigr)$ is the only diagonal matrix
\[
\rho\bigl(H^{00}\bigr)= \mathrm{diag}(1,1,1,-1,-1,-1),
\]
i.e., the basis vectors $ \ket{00}, \ket{11} $ and $\ket{22}$ are degenerate highest weight vectors of $H^{00}$ with highest weight $+1$, while $ \ket{02}, \ket{21}$ and $\ket{10}$ are degenerate lowest weight vectors of $H^{00}$ with lowest weight $-1$.
The representation matrices for the remaining basis elements are given by
\[
	\rho\bigl(H^{11}\bigr) =
	\begin{pNiceArray}{ccc|ccc}
		0 & 0 & \xi^2 & & &
		\\
		1 & 0 & 0 & & &
		\\
		0 & \xi & 0 & & &
		\\ \hline
		& & & 0 & -\xi & 0
		\\
		& & & 0 & 0 & -1
		\\
		& & & -\xi^2 & 0 & 0
	\end{pNiceArray},
	 \qquad
	\rho\bigl(H^{22}\bigr) =
	\begin{pNiceArray}{ccc|ccc}
		0 & \xi^2 & 0 & & &
		\\
		0 & 0 & \xi & & &
		\\
		1 & 0 & 0 & & &
		\\ \hline
		& & & 0 & 0 & -1
		\\
		& & & -\xi & 0 & 0
		\\
		& & & 0 & -\xi^2 & 0
	\end{pNiceArray},
\]
and
\begin{alignat*}{3}
&	\rho\bigl(E_+^{01}\bigr)=
	\begin{pNiceArray}{ccc|ccc}
		& & & 1 & 0 & 0
		\\
		& & & 0 & 0 & 1
		\\
		& & & 0 & 1 & 0
		\\ \hline
		& & & & &
		\\
		& & & & &
		\\
		& & & & &
	\end{pNiceArray},
	\qquad&&
	\rho\bigl(E_-^{02}\bigr)=
	\begin{pNiceArray}{ccc|ccc}
		& & & & &
		\\
		& & & & &
		\\
		& & & & &
		\\ \hline
		1 & 0 & 0 & &
		\\
		0 & 0 & 1 & &
		\\
		0 & 1 & 0 & &
	\end{pNiceArray},& \\
	&\rho\bigl(E_+^{12}\bigr)=
	\begin{pNiceArray}{ccc|ccc}
		& & & 0 & \xi & 0
		\\
		& & & \xi^2 & 0 & 0
		\\
		& & & 0 & 0 & 1
		\\ \hline
		& & & & &
		\\
		& & & & &
		\\
		& & & & &
	\end{pNiceArray},
	\qquad &&
	\rho\bigl(E_-^{21}\bigr)=
	\begin{pNiceArray}{ccc|ccc}
		& & & & &
		\\
		& & & & &
		\\
		& & & & &
		\\ \hline
		0 & \xi^2 & 0 & &
		\\
		1 & 0 & 0 & &
		\\
		0 & 0 & \xi & &
	\end{pNiceArray},&
	\notag \\
&	\rho\bigl(E_+^{20}\bigr)=
	\begin{pNiceArray}{ccc|ccc}
		& & & 0 & 0 & 1
		\\
		& & & 0 & \xi^2 & 0
		\\
		& & & \xi & 0 & 0
		\\ \hline
		& & & & &
		\\
		& & & & &
		\\
		& & & & &
	\end{pNiceArray},
	\qquad &&
	\rho\bigl(E_-^{10}\bigr)=
	\begin{pNiceArray}{ccc|ccc}
		& & & & &
		\\
		& & & & &
		\\
		& & & & &
		\\ \hline
		0 & 0 & \xi^2 & &
		\\
		0 & \xi & 0 & &
		\\
		1 & 0 & 0 & &
	\end{pNiceArray},&
\end{alignat*}
where an empty block denotes the zero matrix.

With these matrices, the following statements are easily verified by direct computation:
\begin{enumerate}\itemsep=0pt
	\item The commutant $M^{\alpha}$ does not exist for $ \alpha \neq 00, 11, 22$, namely, $M^{\alpha}$ is the zero matrix for these degrees.
	\item For $ \alpha = 00, 11, 22,$ there exists a unique (up to overall constant) commutant $M^{\alpha}$. $M^{00}$ is the identity matrix, while $M^{11}$, $M^{22}$ are given by
\begin{equation*}
	M^{11} =
	\begin{pNiceArray}{ccc|ccc}
		0 & 0 & 1 & &
		\\
		\xi & 0 & 0 & &
		\\
		0 & \xi^2 & 0 & &
		\\ \hline
		& & & 0 & \xi^2 & 0
		\\
		& & & 0 & 0 & \xi
		\\
		& & & 1 & 0 & 0
	\end{pNiceArray},
	\qquad
	M^{22} =
	\begin{pNiceArray}{ccc|ccc}
		0 & 1 & 0 & &
		\\
		0 & 0 & \xi^2 & &
		\\
		\xi & 0 & 0 & &
		\\ \hline
		& & & 0 & 0 & \xi
		\\
		& & & \xi^2 & 0 & 0
		\\
		& & & 0 & 1 & 0
	\end{pNiceArray}.
\end{equation*}
\end{enumerate}

It follows that invariant bilinear forms of degrees $00$, $11$ and $22$ exist.
As in the previous subsection, we compute the normalized invariant bilinear forms:
\begin{align*}
	&\eta^{00}(A,B)= \frac{1}{6}\operatorname{tr}(\rho(A)\rho(B)),\qquad
	\eta^{11}(A,B) = \frac{1}{6}\operatorname{tr}\bigl(\rho(A)M^{22}\rho(B)\bigr),
	\notag \\
&	\eta^{22}(A,B)= \frac{1}{6}\operatorname{tr}\bigl(\rho(A)M^{11}\rho(B)\bigr),\qquad
	A, B \in \mathbb{Z}_3^2\text{-}\sl(2)
\end{align*}
We list the nonvanishing components of $\eta^{\alpha}$ below:
\begin{alignat*}{3}
	&\eta^{00}\bigl(H^{00},H^{00}\bigr)= 1, \qquad&&
	\eta^{00}\bigl(H^{11},H^{22}\bigr)= \eta^{00}\bigl(H^{22},H^{11}\bigr) = \xi^2,&
	\notag \\
	&\eta^{00}\bigl(E_+^{01},E_-^{02}\bigr)= \eta^{00}\bigl(E_-^{02},E_+^{01}\bigr) = \frac{1}{2},
	\qquad &&
	\eta^{00}\bigl(E_+^{20},E_-^{10}\bigr)= \eta^{00}\bigl(E_-^{10},E_+^{20}\bigr) = \frac{1}{2},&
	\notag \\
	&\eta^{00}\bigl(E_+^{12},E_-^{21}\bigr)= \eta^{00}\bigl(E_-^{21},E_+^{12}\bigr) = \frac{\xi}{2}.\qquad &&&
\end{alignat*}
The nonvanishing components of $\eta^{11}$ are
\begin{alignat*}{3}
	&\eta^{11}\bigl(H^{00},H^{11}\bigr)= \eta^{11}\bigl(H^{11},H^{00}\bigr) = 1, \qquad&&
	\eta^{11}\bigl(H^{22},H^{22}\bigr)= \xi,&
	\notag \\
	&\eta^{11}\bigl(E_+^{12},E_-^{02}\bigr)= \eta^{11}\bigl(E_+^{20},E_-^{21}\bigr) = \frac{1}{2},
	\qquad &&
	\eta^{11}\bigl(E_+^{01},E_-^{10}\bigr)= \frac{\xi}{2},&
	\notag \\
	&\eta^{11}\bigl(E_-^{02},E_+^{12}\bigr)= \eta^{11}\bigl(E_-^{21},E_+^{20}\bigr) = \frac{\xi^2}{2}, \qquad &&
	\eta^{11}\bigl(E_-^{10},E_+^{01}\bigr)= \frac{1}{2},&
\end{alignat*}
and those of $\eta^{22}$ are
\begin{alignat*}{3}
	&\eta^{22}\bigl(H^{00}, H^{22}\bigr)= \eta^{22}\bigl(H^{22}, H^{00}\bigr) = 1,
	\qquad&&
	\eta^{22}\bigl(H^{11}, H^{11}\bigr)= \xi,&
	\notag \\
	&\eta^{22}\bigl(E_+^{01}, E_-^{21}\bigr)= \eta^{22}\bigl(E_+^{12}, E_-^{10}\bigr) = \frac{\xi^2}{2},
	\qquad &&
	\eta^{22}\bigl(E_+^{20}, E_-^{02}\bigr)= \frac{1}{2},&
	\notag \\
	&\eta^{22}\bigl(E_-^{21}, E_+^{01}\bigr)= \eta^{22}\bigl(E_-^{10}, E_+^{12}\bigr) = \frac{1}{2},
	\qquad &&
	\eta^{22}\bigl(E_-^{02}, E_+^{20}\bigr)= \frac{\xi}{2}.&
\end{alignat*}

The bilinear forms $\eta^{\alpha}$ are nondegenerate.
Their inverses are easily obtained; only the nonvanishing components are listed below:
\begin{alignat*}{3}
	&\eta_{00}\bigl(H^{00},H^{00}\bigr)= 1, \qquad&&
	\eta_{00}\bigl(H^{11},H^{22}\bigr)= \eta_{00}\bigl(H^{22},H^{11}\bigr) = \xi,&
	\notag \\
	&\eta_{00}\bigl(E_+^{01},E_-^{02}\bigr)= \eta_{00}\bigl(E_-^{02},E_+^{01}\bigr) = 2,
	\qquad & &
	\eta_{00}\bigl(E_+^{12},E_-^{21}\bigr)= 2\xi^2,&
	\notag \\
	&\eta_{00}\bigl(E_+^{20},E_-^{10}\bigr)= \eta_{00}\bigl(E_-^{10},E_+^{20}\bigr) = 2,
	\qquad &&
	\eta_{00}\bigl(E_-^{21},E_+^{12}\bigr)= 2\xi^2,&
\end{alignat*}
and
\begin{alignat*}{3}
	&\eta_{11}\bigl(H^{00}, H^{11}\bigr)= \eta_{11}\bigl(H^{11}, H^{00}\bigr) = 1,
	\qquad&&
	\eta_{11}\bigl(H^{22}, H^{22}\bigr)= \xi^2,&
	\notag \\
	&\eta_{11}\bigl(E_+^{12},E_-^{02}\bigr)= \eta_{11}\bigl(E_+^{20},E_-^{21}\bigr) = 2\xi,
	\qquad &&
	\eta_{11}\bigl(E_+^{01},E_-^{10}\bigr)= 2,&
	\notag \\
	&\eta_{11}\bigl(E_-^{02},E_+^{12}\bigr)= \eta_{11}\bigl(E_-^{21},E_+^{20}\bigr) =2,
	\qquad &&
	\eta_{11}\bigl(E_-^{10},E_+^{01}\bigr)= 2\xi^2,&
\end{alignat*}
and
\begin{alignat*}{3}
	&\eta_{22}\bigl(H^{00}, H^{22}\bigr)= \eta_{22}\bigl(H^{22}, H^{00}\bigr) = 1,
	\qquad&&
	\eta_{22}\bigl(H^{11}, H^{11}\bigr)= \xi^2,&
	\notag \\
	&\eta_{22}\bigl(E_+^{01}, E_-^{21}\bigr)= \eta_{22}\bigl(E_+^{12}, E_-^{10}\bigr) = 2,
	\qquad&&
	\eta_{22}\bigl(E_+^{20}, E_-^{02}\bigr)= 2\xi^2,&
	\notag \\
	&\eta_{22}\bigl(E_-^{21}, E_+^{01}\bigr)= \eta_{22}\bigl(E_-^{10}, E_+^{12}\bigr) = 2\xi,
	\qquad &&
	\eta_{22}\bigl(E_-^{02}, E_+^{20}\bigr)= 2.&
\end{alignat*}

We are now ready to write down the graded 2nd order Casimir elements of \Zsl, which are given by
\begin{align*}
	&C_{00}= \bigl(H^{00}\bigr)^2 + \xi \bigl\{ H^{11},H^{22}\bigr\} + 2 \bigl\{E_+^{01},E_-^{02} \bigr\} + 2\bigl\{E_+^{20}, E_-^{10} \bigr\} + 2\xi^2\bigl\{ E_+^{12},E_-^{21} \bigr\},
	\notag \\
	&C_{11}= \bigl\{H^{00}, H^{11} \bigr\} + \xi^2 \bigl(H^{22}\bigr)^2 + 2\bigl(\xi E_+^{12}E_-^{02} + E_-^{02}E_+^{12} +\xi E_+^{20} E_-^{21} + E_-^{21} E_+^{20}
	\notag \\
	&\hphantom{C_{11}=}{} + E_+^{01} E_-^{10} + \xi^2 E_-^{10} E_+^{01}\bigr),
	\notag \\
	&C_{22}= \bigl\{ H^{00}, H^{22} \bigr\} + \xi^2 \bigl(H^{11}\bigr)^2 + 2\bigl(E_+^{01} E_-^{21} + \xi E_-^{21} E_+^{01} + E_+^{12} E_-^{10} + \xi E_-^{10} E_+^{12}
	\notag \\
	&\hphantom{C_{22}=}{} + \xi^2 E_+^{20} E_-^{02} + E_-^{02} E_+^{20}\bigr),
\end{align*}
where $\{ \ , \ \}$ denotes the anticommutator.

We now turn to the loop algebra of \Zsl, whose basis is denoted by
\begin{equation*}
	H^{\alpha (m)}, \quad E_+^{\beta (m)}, \quad E_-^{\gamma (m)}, \qquad m \in \mathbb{Z},
\end{equation*}
where $\alpha \in \{ 00, 11, 22 \}$, $\beta \in \{ 01, 12, 20 \}$, $\gamma \in \{ 02, 21, 10 \} $.
Following Theorem \ref{THM:centralext}, the loop algebra admits three central extensions of degrees $00$, $11$ and $22$.
The nonvanishing color Lie brackets are given as follows:
\begin{alignat*}{3}
	&\bigl\llbracket H^{\alpha (m)}, E_{\pm}^{\beta (n)} \bigr\rrbracket= \pm 2\sigma(\alpha,\beta) E_{\pm}^{\alpha+\beta (m+n)},
	\qquad &&& \notag \\
	&\bigl\llbracket H^{00 (m)}, H^{00 (n)} \bigr\rrbracket= m \delta_{m+n,0} c_{00}, \qquad&&
	\bigl\llbracket H^{11 (m)}, H^{22 (n)} \bigr\rrbracket= m \delta_{m+n,0} \xi^2 c_{00},&
	\notag \\
	&\bigl\llbracket H^{00 (m)}, H^{11 (n)} \bigr\rrbracket= m \delta_{m+n,0} c_{11},
	\qquad &&
	\bigl\llbracket H^{22 (m)}, H^{22 (n)} \bigr\rrbracket= m \delta_{m+n,0} \xi c_{11},&
	\notag \\
	&\bigl\llbracket H^{00 (m)}, H^{22 (n)} \bigr\rrbracket= m \delta_{m+n,0} c_{22},\qquad
	&&
	\bigl\llbracket H^{11 (m)}, H^{11 (n)} \bigr\rrbracket= m \delta_{m+n,0} \xi c_{22},&
\end{alignat*}
and
\begin{align*}
	&\bigl\llbracket E_{+}^{01 (m)}, E_{-}^{02 (n)} \bigr\rrbracket= \bigl\llbracket E_+^{20 (m)}, E_-^{10 (n)} \bigr\rrbracket = H^{00 (m+n)} + \frac{m}{2} \delta_{m+n,0} c_{00},
	\notag \\
	&\bigl\llbracket E_{+}^{12 (m)}, E_{-}^{21 (n)} \bigr\rrbracket= \xi H^{00 (m+n)} + \frac{m \xi}{2} \delta_{m+n,0} c_{00},
	\notag \\
	&\bigl\llbracket E_{+}^{01 (m)}, E_{-}^{10 (n)} \bigr\rrbracket= H^{11 (m+n)} + \frac{m}{2} \delta_{m+n,0} c_{11},
	\notag \\
	&\bigl\llbracket E_{+}^{12 (m)}, E_{-}^{02 (n)} \bigr\rrbracket=\bigl\llbracket E_{+}^{20 (m)}, E_{-}^{21 (n)} \bigr\rrbracket = \xi^2 H^{11 (m+n)} + \frac{m\xi^2}{2} \delta_{m+n,0}c_{11},
	\notag \\
	&\bigl\llbracket E_{+}^{01 (m)}, E_{-}^{21 (n)} \bigr\rrbracket= \bigl\llbracket E_{+}^{12 (m)}, E_{-}^{10 (n)} \bigr\rrbracket = H^{22 (m+n)} + \frac{m}{2} \delta_{m+n,0}c_{22},
	\notag \\
	&\bigl\llbracket E_{+}^{20 (m)}, E_{-}^{02 (n)} \bigr\rrbracket= \xi H^{22 (m+n)} + \frac{m\xi}{2} \delta_{m+n,0}c_{22}.
\end{align*}

\section[Z\_2\^{}2-graded extension of osp(m|2n)]{$\boldsymbol{\Z2}$-graded extension of $\boldsymbol{\osp(m|2n)}$} \label{SEC:osp}

In this section, we consider a particular $\Z2$-graded extension of the Lie superalgebra $\osp(m|2n)$ and show that the color Lie algebra has a commutatnt of degree $11.$
The grading group $\Gamma$ and the commutative factor $\omega(\alpha,\beta) = (-1)^{\alpha_1 \beta_1 + \alpha_2\beta_2}$ are the same as in Section~\ref{SEC:qn}.

\subsection{Definition, commutants and basis}

We define the $\Z2$-graded color algebra, denoted by \Zosp, as a subalgebra of algebra $ \gl(m,m|2n,2n)$.
An element of $\mathfrak{gl}(m,m|2n,2n)$ is given by a $(2m+4n) \times (2m+4n)$ matrix with a $4 \times 4$ block structure, where each block has a definite $\Z2$-degree. We consider a particular class of elements of $\mathfrak{gl}(m,m|2n,2n)$ of the form
\begin{equation}
	X =
	\begin{pNiceMatrix}[first-row][first-col]
		& m & m & 2n & 2n \\
		m & A^{00} & A^{11} & A^{01} & A^{10}
		\\
		m & A^{11} & A^{00} & A^{10} & A^{01}
		\\
		2n & C^{01} & C^{10} & C^{00} & C^{11}
		\\
		2n & C^{10} & C^{01} & C^{11} & C^{00}
	\end{pNiceMatrix},
	\label{doubledGL}
\end{equation}
where the superscripts indicate the $\Z2$-degree of each block, and each block matrix appears twice.

The bilinear form defining \Zosp\ is given by a $ 4 \times 4$ block matrix of the same size as~$X$:
\begin{equation*}
	J =
	\begin{pmatrix}
		{\rm i}B_m & 0 & 0 & 0
		\\
		0 & {\rm i}B_m & 0 & 0
		\\
		0 & 0 & G_{2n} & 0
		\\
		0 & 0 & 0 & G_{2n}
	\end{pmatrix}.
\end{equation*}
Here,
\begin{align*}
	&B_m=
	\begin{pmatrix}
		0 & \mathbb{I}_{\ell} \\ \mathbb{I}_{\ell} & 0
	\end{pmatrix}, \quad m = 2\ell,
	\qquad
	B_m =
	\begin{pmatrix}
		0 & \mathbb{I}_{\ell} & 0 \\ \mathbb{I}_{\ell} & 0 & 0
		\\
		0 & 0 & 1
	\end{pmatrix}, \quad m = 2\ell+1,
	\notag \\
	&G_{2n}=
	\begin{pmatrix}
		0 & -\mathbb{I}_{n} \\ \mathbb{I}_{n} & 0
	\end{pmatrix},
\end{align*}
where $ \mathbb{I}_{\ell} $ denotes the $\ell \times \ell$ identity matrix.

To define \Zosp, we impose the following conditions on the matrix $X$ given in \eqref{doubledGL}:
\begin{alignat}{3}
&	X^{00}J + J \bigl(X^{00}\bigr)^{\mathsf{T}}= 0,\qquad &&
	X^{11}J + J \bigl(X^{11}\bigr)^{\mathsf{T}}= 0,&
	\notag\\
&	X^{01}J + {\rm i}J \bigl(X^{01}\bigr)^{\mathsf{T}}= 0,
	\qquad &&
	X^{10}J -{\rm i} J \bigl(X^{10}\bigr)^{\mathsf{T}}= 0.& \label{ospCondition}
\end{alignat}
Here $X^{\alpha}$, with $\alpha \in \mathbb{Z}_2$, denotes the matrix in~\eqref{doubledGL} whose entries are nonzero only in the blocks of $\mathbb{Z}_2^2$-degree $\alpha$; e.g.,
\begin{equation*}
	X^{11} =
	\begin{pmatrix}
		0 & A^{11} & 0 & 0
		\\
		A^{11} & 0 & 0 & 0
		\\
		0 & 0 & 0 & C^{11}
		\\
		0 & 0 & C^{11} & 0
	\end{pmatrix}
\end{equation*}
and $ X^{\mathsf{T}} $ denotes the transpose of $X.$
It is straightforward to check by direct computation that if $X^{\alpha}$ and $Y^{\beta}$ satisfy the conditions in \eqref{ospCondition}, then so does $\llbracket X^{\alpha}, Y^{\beta} \rrbracket$.

The conditions in \eqref{ospCondition} yield the following constraints on the matrix blocks:
\begin{alignat}{3}
	&A^{00} B_m+ B_m \bigl(A^{00}\bigr)^{\mathsf{T}} = 0, \qquad && A^{11} B_m+ B_m \bigl(A^{11}\bigr)^{\mathsf{T}} = 0,&
	\notag\\
&	C^{00}G_{2n}+ G_{2n}\bigl(C^{00}\bigr)^{\mathsf{T}} = 0,\qquad&& C^{11} G_{2n}+ G_{2n}\bigl(C^{11}\bigr)^{\mathsf{T}} = 0,&
	\notag \\
&	C^{01}= -G_{2n}\bigl(A^{01}\bigr)^{\mathsf{T}} B_m, \qquad && C^{10}= G_{2n}\bigl(A^{10}\bigr)^{\mathsf{T}} B_m.&
	\label{StrOfX}
\end{alignat}
It follows that
\begin{equation*}
	A^{00}, A^{11} \in \so(m), \qquad C^{00}, C^{11} \in \sp(2n)
\end{equation*}
and $ A^{01}$, $A^{10}$ are arbitrary $ m \times 2n$ matrices.
Therefore, we see that
\[
\dim \Z2\text{-}\osp(m|2n) = (m+2n)^2-m+2n,
\]
which is twice the dimension of the Lie superalgebra $\osp(m|2n)$.

We comment on preceding works on $\Z2$-graded extensions of $\osp(m|2n)$ in turn:
\begin{itemize}\itemsep=0pt
	\item[(1)] Rittenberg and Wyler introduced a $\Z2$-graded extension of the Lie supergroup $Osp(m|2n)$ and define a $\Z2$-graded $\osp(m|2n)$ as its infinitesimal transformation \cite{rw2}.
	Algebraic definitions of a $\Z2$-graded extension of $\osp(m|2n)$ can be found in \cite{GrJav, JaYangWyb,SVdJ2018,SVdJ2024}.
	However, none of these works discuss the particular subclass defined by \eqref{doubledGL}.
	\item[(2)] A $\Z2$-graded extension of the type defined in \eqref{doubledGL}, which corresponds to $\Z2$-$\osp(1|2)$ in the present notation, first appeared as a $\Z2$-graded extension of conformal quantum mechanics~\cite{Aizawa_2020}.
	The same algebra was also obtained via a module construction of color Lie algebras~\cite{LU20231}.
	It was pointed out in \cite{AiSe} that $\Z2$-$\osp(1|2)$ has a $11$-graded commutant.
	\item[(3)] The lowest weight representations of $\Z2$-$\osp(1|2)$ are classified in \cite{AmaAi} (see \cite{AlHoSe} for other representations of $\Z2$-$\osp(1|2)$).
\end{itemize}

Now we return to the search for commutants of \Zosp.

\begin{Proposition} \label{PROP:comm}
	There exists a unique $($up to constant multiple$)$ $11$-graded commutant of \Zosp, which is given by
	\begin{equation}
		M^{11} =
		\begin{pNiceArray}{cc|cc}
			0 & \mathbb{I}_m & &
			\\
			\mathbb{I}_m & 0 & &
			\\ \hline
			& & 0 & -\mathbb{I}_{2n}
			\\
			& & -\mathbb{I}_{2n} & 0
		\end{pNiceArray}.
		\label{M11osp}
	\end{equation}
	On the other hand, there exist no $01$- or $10$-graded commutants.
\end{Proposition}

\begin{proof}
	$M^{11}$ must be of the form
	\begin{equation*}
		M^{11} =
		\begin{pNiceArray}{cc|cc}
			0 & N_1 & &
			\\
			N_2 & 0 & &
			\\ \hline
			& & 0 & N_3
			\\
			& & N_4 & 0
		\end{pNiceArray}.
	\end{equation*}
	The condition $[X^{00}, M^{11}] = 0$ yields the following constraints on $N_k$:
	\begin{equation*}
		\bigl[A^{00},N_1\bigr] = \bigl[A^{00},N_2\bigr] = 0, \qquad \bigl[C^{00},N_3\bigr] = \bigl[C^{00},N_4\bigr] = 0.
	\end{equation*}
	Recalling that $ A^{00} \in \so(m) $ and $ C^{00} \in \sp(2n)$, we obtain, by Schur's lemma,
	\begin{equation*}
		N_1 = \alpha_1 \mathbb{I}_m, \qquad N_2 = \alpha_2 \mathbb{I}_m, \qquad
		N_3 = \alpha_3 \mathbb{I}_{2n}, \qquad N_4 = \alpha_4 \mathbb{I}_{2n}, \qquad \alpha_k \in \mathbb{C}.
	\end{equation*}
	The condition $[X^{11}, M^{11}] = 0$ further implies that $\alpha_1 = \alpha_2$ and $\alpha_3 = \alpha_4$. Finally, the condition $\bigl\{ X^{01}, M^{11}\bigr\} = 0$ yields $ \alpha_3 = -\alpha_1$, and this relation is consistent with $\bigl\{ X^{10}, M^{11}\bigr\} = 0$.
	Thus, $M^{11}$ is given by~\eqref{M11osp}.

	We can search for $M^{01}$ in a similar way. We write
\begin{equation*}
	M^{01} =
	\begin{pNiceArray}{cc|cc}
		& & N_1 & 0
		\\
		& & 0 & N_2 \\ \hline
		N_3 & 0 & &
		\\
		0 & N_4 & &
	\end{pNiceArray}
\end{equation*}
and the condition $[X^{00}, M^{01}] = 0$ yields the constraints
\[
	A^{00} N_1 = N_1 C^{00}, \qquad A^{00} N_2 = N_2 C^{00}, \qquad
	C^{00} N_3 = N_3 A^{00}, \qquad C^{00} N_4 = N_4 A^{00}.
\]
If $ A^{00} = 0 $ and $ C^{00} \neq 0$, then $ N_k \equiv 0$, and hence $ M^{01} = 0.$
In the same way, one finds that $M^{10} = 0$.
\end{proof}

From Proposition~\ref{PROP:comm}, we see that \Zosp\ has $00$- and $11$-graded 2nd order Casimir elements.
To derive explicit formulas for them, we first need to specify a basis of \Zosp.
This is also helpful to understand structures of the color Lie algebra.
Recall that $ X \in $ \Zosp\ has a $ 4 \times 4 $ block structure, see \eqref{doubledGL}, and that each block further has the following substructure.
In the case of $ m = 2\ell+1$,
\begin{equation*}
	A^{00} =
	\begin{pmatrix}
		L & Y & x
		\\
		W & -L^{\mathsf{T}} & y
		\\
		-y^{\mathsf{T}} & -x^{\mathsf{T}} & 0
	\end{pmatrix}
	\in \so(2\ell+1), \qquad
	C^{00} =
	\begin{pmatrix}
		\tilde{L} & \tilde{Y} \\ \tilde{W} & -\tilde{L}^{\mathsf{T}}
	\end{pmatrix}
	\in \sp(2n),
\end{equation*}
where $ L \in \gl(\ell)$ and $ Y$, $W $ are skew-symmetric $\ell \times \ell $ matrices, and $ x, y \in \mathbb{C}^{\ell}$,
while $\tilde{L} \in \gl(n)$ and $ \tilde{Y}$, $\tilde{W}$ are symmetric $ n \times n $ matrices.
The block structures of $A^{01}$ and $A^{10}$ are induced from this structure:
\begin{equation*}
	A^{01} =
	\begin{pmatrix}
		P & Q \\ R & S \\
		a & b
	\end{pmatrix},
	\qquad
	A^{10} =
	\begin{pmatrix}
		\tilde{P} & \tilde{Q} \\ \tilde{R} & \tilde{S} \\
		\tilde{a} & \tilde{b}
	\end{pmatrix},
\end{equation*}
where $ P$, $Q$, $\tilde{P}$ and $\tilde{Q} $ are $ \ell \times n $ matrices, and $ a$, $b$, $\tilde{a} $ and $\tilde{b} $ are $n$-dimensional row vectors.
For $m = 2\ell$, the corresponding changes are obvious.
Then, \eqref{StrOfX} determines $ C^{01}$ and $ C^{10}$ as follows:%
\begin{equation*}
	C^{01}
	=
	\begin{pmatrix}
		S^{\mathsf{T}} & Q^{\mathsf{T}} & b^{\mathsf{T}} \\
		-R^{\mathsf{T}} & -P^{\mathsf{T}} & -a^{\mathsf{T}}
	\end{pmatrix},
	\qquad
	C^{10} =
	\begin{pmatrix}
		-\tilde{S}^{\mathsf{T}} & - \tilde{Q}^{\mathsf{T}} & -\tilde{b}^{\mathsf{T}}
		\\
		\tilde{R}^{\mathsf{T}} & \tilde{P}^{\mathsf{T}} & \tilde{a}^{\mathsf{T}}
	\end{pmatrix}.
\end{equation*}
Therefore, we introduce a concise notation for matrix units, by which the position of a nonvanishing entry can be easily read off.

Let $ e_{ab}$ be the matrix unit of size $(2m+4n)\times (2m+4n)$.
From now on, Latin indices take values in $\{ 1, 2, \dots, \ell= \mathrm{rank}\ \so(m) \}$, while Greek indices take values in $\{ 1, 2, \dots, n = \mathrm{rank}\ \sp(2n) \}$.
We~introduce the following notation for the matrix units whose nonvanishing entries lie in the first $\ell $ rows of the matrix $ X $ in \eqref{doubledGL}: For $ m = 2\ell$,
\begin{alignat*}{5}
	&E_{ij}:= e_{ij}, \qquad&& E_{ij'}:= e_{i\, \ell+j}, \qquad&& {E_i}^j:= e_{i\, m+j},\qquad&& {E_{i}}^{j'}:= e_{i\, m+\ell+j},&
	\notag\\
&	E_{i\nu}:=e_{i\, 2m+\nu}, \qquad && E_{i\nu'}:= e_{i\, 2m+n+\nu}, \qquad && {E_i}^{\nu}:= e_{i\, 2m+2n+\nu}, \qquad && {E_i}^{\nu'}:= e_{i\, 2m+3n+\nu},&
\end{alignat*}
and we need two more for $m = 2\ell+1$:
\begin{equation*}
	E_{im} := e_{i\, 2\ell+1}, \qquad {E_i}^m := e_{i,4\ell+2}.
\end{equation*}
The meaning of this notation is clear; $E_{ij}$ has its nonvanishing entry in the $\ell \times \ell$ block at the top left corner of $X$.
As the second index changes, the position of the nonvanishing entry shifts to the right.
The same rule can be applied to the first index (with shifts downward), allowing us to introduce a notation for the matrix units whose nonvanishing entries lie in the blocks: 
\begin{equation*}\renewcommand{\arraystretch}{1.3}
		\begin{pNiceArray}{cc:c|cc:c|cc|cc}[first-row,first-col]
			& \ell & \ell & 1 & \ell & \ell & 1 & n & n & n & n
			\\
			\ell & E_{ij} & E_{ij'} & E_{im} & {E_i}^{j} & {E_i}^{j'} & {E_{i}}^m & E_{i\nu} & E_{i\nu'} & {E_i}^{\nu} & {E_i}^{\nu'} \\
			\ell & E_{i'j} & E_{i'j'} & E_{i'm} & {E_{i'}}^{j} & {E_{i'}}^{j'} & {E_{i'}}^m & E_{i'\nu} & E_{i'\nu'} & {E_{i'}}^{\nu} & {E_{i'}}^{\nu'} \\ \hdashline
			1 & E_{mj} & E_{mj'} & E_{mm} & {E_m}^{j} & {E_m}^{j'} & {E_{m}}^m & E_{m\nu} & E_{m\nu'} & {E_m}^{\nu} & {E_m}^{\nu'} \\ \hline
			\ell &{E^i}_{j} & {E^i}_{j'} & {E^{i}}_{m} & E^{ij} & E^{ij'} & E^{im} & {E^i}_{\nu} & {E^i}_{\nu'} & E^{i\nu} & E^{i\nu'} \\
			\ell &{E^{i'}}_{j} & {E^{i'}}_{j'} & {E^{i'}}_{m} & E^{i'j} & E^{i'j'} & E^{i'm} & {E^{i'}}_{\nu} & {E^{i'}}_{\nu'} & E^{i'\nu} & E^{i'\nu'} \\ \hdashline
			1 & {E^m}_{j} & {E^m}_{j'} & {E^m}_{m} & E^{mj} & E^{mj'} & E^{mm} & {E^{m}}_{\nu} & {E^{m}}_{\nu'} & E^{m\nu} & E^{m\nu'} \\ \hline
			n & E_{\mu j} & E_{\mu j'} & E_{\mu m} & {E_\mu}^{j} & {E_\mu}^{j'} & {E_{\mu}}^m & E_{\mu\nu} & E_{\mu\nu'} & {E_\mu}^{\nu} & {E_\mu}^{\nu'} \\
			n & E_{\mu'j} & E_{\mu'j'} & E_{\mu' m} & {E_{\mu'}}^{j} & {E_{\mu'}}^{j'} & {E_{\mu'}}^m & E_{\mu'\nu} & E_{\mu'\nu'} & {E_{\mu'}}^{\nu} & {E_{\mu'}}^{\nu'} \\ \hline
			n & {E^\mu}_{j} & {E^\mu}_{j'} & {E^{\mu}}_{m} & E^{\mu j} & E^{\mu j'} & E^{\mu m} & {E^\mu}_{\nu} & {E^\mu}_{\nu'} & E^{\mu\nu} & E^{\mu\nu'} \\
			n & {E^{\mu'}}_{j} & {E^{\mu'}}_{j'} & {E^{\mu'}}_{m} & E^{\mu'j} & E^{\mu'j'} & E^{\mu' m} & {E^{\mu'}}_{\nu} & {E^{\mu'}}_{\nu'} & E^{\mu'\nu} & E^{\mu'\nu'}
		\end{pNiceArray}.
	\end{equation*}
Note that the prime on $m$ is omitted for simplicity.
One advantage of this notation is the simplicity of the product: if the adjacent indices are not of the same type, the product vanishes. For example,
\begin{equation*}
	E_{ij'} E_{pq} = {E_i}^j E^{\mu \nu} = {E_i}^j {E_p}^q = 0, \qquad
	E_{ij}{E_p}^q = \delta_{jp} {E_i}^q.
\end{equation*}

Using this notation of the block structures described above, we choose the following basis of \Zosp\ in the case of $ m= 2\ell + 1$:
\begin{enumerate}\itemsep=0pt
\item[(i)] basis of degree $00$
\begin{align*}
	&T_{ij}=E_{ij}-E_{j'i'}+E^{ij}-E^{j'i'}, \qquad 1 \leq i, j \leq \ell, \text{ or } j=m,
	\notag\\
	&T_{ij'}=E_{ij'}-E_{ji'}+E^{ij'}-E^{ji'},\qquad 1 \leq i< j \leq \ell,
	\notag\\
	&T_{i'j}=E_{i'j}-E_{j'i}+E^{i'j}-E^{j'i},\qquad 1 \leq i< j \leq \ell, \text{ or } j=m,
	\notag\\
	&T_{\mu\nu}=E_{\mu\nu}-E_{\nu'\mu'}+E^{\mu\nu}-E^{\nu'\mu'}, \qquad 1 \leq \mu, \nu \leq n,
	\notag\\
	&T_{\mu\nu'}=
	\begin{cases}
		E_{\mu\nu'}+E_{\nu\mu'}+E^{\mu\nu'}+E^{\nu\mu'},& 1 \leq \mu<\nu \leq n,\\
		E_{\mu\mu'}+E^{\mu\mu'},&\mu=\nu,
	\end{cases}
	\notag\\
	&T_{\mu'\nu}=
	\begin{cases}
		E_{\mu'\nu}+E_{\nu'\mu}+E^{\mu'\nu}+E^{\nu'\mu},& 1 \leq \mu<\nu \leq n,\\
		E_{\mu'\mu}+E^{\mu'\mu}, & \mu=\nu.
	\end{cases}
\end{align*}
\item[(ii)] basis of degree $11$
\begin{align*}
	&{U_{i}}^{j}={E_{i}}^{j}-{E_{j'}}^{i'}+{E^{i}}_{j}-{E^{j'}}_{i'}, \qquad 1 \leq i, j \leq \ell, \text{ or } j=m,
	\notag\\
	&{U_{i}}^{j'}={E_{i}}^{j'}-{E_{j}}^{i'}+{E^{i}}_{j'}-{E^{j}}_{i'},\qquad 1 \leq i< j \leq \ell,
	\notag\\
	&{U_{i'}}^{j}={E_{i'}}^{j}-{E_{j'}}^{i}+{E^{i'}}_{j}-{E^{j'}}_{i},\qquad 1 \leq i< j \leq \ell, \text{ or } j=m,
	\notag\\
	&{U_{\mu}}^{\nu}={E_{\mu}}^{\nu}-{E_{\nu'}}^{\mu'}+{E^{\mu}}_{\nu}-{E^{\nu'}}_{\mu'},
	\notag\\
	&{U_{\mu}}^{\nu'}=
	\begin{cases}
		{E_{\mu}}^{\nu'}+{E_{\nu}}^{\mu'}+{E^{\mu}}_{\nu'}+{E^{\nu}}_{\mu'},& 1 \leq \mu<\nu \leq n,\\
		{E_{\mu}}^{\mu'}+{E^{\mu}}_{\mu'},& \mu=\nu,
	\end{cases}
	\notag\\
	&{U_{\mu'}}^{\nu}=
	\begin{cases}
		{E_{\mu'}}^{\nu}+{E_{\nu'}}^{\mu}+{E^{\mu'}}_{\nu}+{E^{\nu'}}_{\mu},& 1 \leq \mu<\nu \leq n,\\
		{E_{\mu'}}^{\mu}+{E^{\mu'}}_{\mu},& \mu=\nu.
	\end{cases}
\end{align*}
\item[(iii)] basis of degree $01$, ($1 \leq i \leq \ell$, $1 \leq \mu \leq n$)
\begin{align*}
	&\Lambda_{i\mu}=E_{i\mu}-E_{\mu'i'}+E^{i\mu}-E^{\mu'i'},
	\notag\\
	&\Lambda_{i\mu'}=E_{i\mu'}+E_{\mu i'}+E^{i\mu'}+E^{\mu i'},
	\notag\\
	&\Lambda_{i'\mu}=E_{i'\mu}-E_{\mu'i}+E^{i'\mu}-E^{\mu'i},
	\notag\\
	&\Lambda_{i'\mu'}=E_{i'\mu'}+E_{\mu i}+E^{i'\mu'}+E^{\mu i}.
\end{align*}
In the first and the second lines, we allow $ i =m$.
\item[(iv)] basis of degree $10$, ($1 \leq i \leq \ell$, $1 \leq \mu \leq n$)
\begin{align*}
&	{\Gamma_{i}}^{\mu}={E_{i}}^{\mu}+{E_{\mu'}}^{i'}+{E^{i}}_{\mu}+{E^{\mu'}}_{i'},
	\notag\\
&	{\Gamma_{i}}^{\mu'}={E_{i}}^{\mu'}-{E_{\mu}}^{i'}+{E^{i}}_{\mu'}-{E^{\mu}}_{i'},
	\notag\\
&	{\Gamma_{i'}}^{\mu}={E_{i'}}^{\mu}+{E_{\mu'}}^{i}+{E^{i'}}_{\mu}+{E^{\mu'}}_{i},
	\notag\\
&	{\Gamma_{i'}}^{\mu'}={E_{i'}}^{\mu'}-{E_{\mu}}^{i}+{E^{i'}}_{\mu'}-{E^{\mu}}_{i}.
\end{align*}
In the first and the second lines, we allow $ i =m$.
The basis for the case $m = 2\ell$ is obtained by eliminating the elements with $i=m $ and $j=m$.
\end{enumerate}

\subsection[Structure of Z\_2\^{}2-osp(m|2n) and 2nd order Casimir elements]{Structure of $\boldsymbol{\Z2}$-$\boldsymbol{\osp(m|2n)}$ and 2nd order Casimir elements}

The Cartan subalgebra of \Zosp\ is defined as the maximal commuting subalgebra of degree $00$.
Obviously, it is spanned by
\begin{alignat}{3}
&	H_i:= T_{ii} = E_{ii} - E_{i'i'} + E^{ii} - E^{i'i'}, \qquad&&
	 1 \leq i \leq \ell ,&
	\nn \\
	&H_{\mu}:= T_{\mu\mu} = E_{\mu\mu} - E_{\mu'\mu'} + E^{\mu\mu} - E^{\mu'\mu'},\qquad
	& & 1 \leq \mu \leq n,&
\end{alignat}
and has dimension $\ell+n$.

Using the Cartan subalgebra, we obtain the root space decomposition of $\g := $\Zosp\ as usual. Let $ \mathfrak{h} $ be the Cartan subalgebra of $\g$ and
\begin{equation*}
	\g^{(\lambda)} = \{ X \in \g \mid [H, X] = \lambda(H) X, \ \forall H \in \mathfrak{h} \}.
\end{equation*}
Note that, since $\deg(\mathfrak{h}) = 00$, the color Lie bracket is identical to the commutator.
First, observe that there exist zero roots of degree $11$:{\samepage
\begin{equation*}
	[H, {U_i}^i] = [H, {U_{\mu}}^{\mu}] = 0, \qquad \forall H \in \mathfrak{h}.
\end{equation*}
Hence, the zero root space has dimension $\ell+n$.}

Denote by $ \epsilon_i$, $\delta_{\mu}$ the basis of the dual space $\mathfrak{h}^*$ such that
\begin{equation*}
	\epsilon_i(H_j) = \delta_{ij}, \qquad
	\delta_{\mu}(H_{\nu}) = \delta_{\mu\nu}.
\end{equation*}
Then, it is not difficult to see that the basis of $\g$ introduced in the previous subsection yields a~basis of the root spaces, e.g.,
\begin{align*}
	&[H_i, T_{jk}]= (\delta_{ij} - \delta_{ik} ) T_{jk} = (\epsilon_j - \epsilon_k)(H_i) T_{jk},
	\nn \\
	&[H_i, {U_j}^k]=(\delta_{ij} - \delta_{ik} ) {U_j}^k = (\epsilon_j - \epsilon_k)(H_i) {U_j}^k,
	\nn \\
	&[H_i, \Lambda_{j\nu}]= \delta_{ij} \Lambda_{j\nu} = \epsilon_j(H_i) \Lambda_{j\nu},
	\nn \\
	&[H_i, {\Gamma_j}^{\nu}]= \delta_{ij}{\Gamma_j}^{\nu} = \epsilon_j(H_i) {\Gamma_j}^{\nu}.
\end{align*}
Repeating similar computations, we see that, apart from the zero root, the root system of $\g$ coincides with that of the ordinary Lie superalgebra~$\osp(m|2n)$.
However, each nonzero root space is two-dimensional and consists of basis elements of different degrees.
The root system of~$\g$, including the zero root, is summarized in the table below. For $m = 2\ell$,
	\begin{equation*}\renewcommand{\arraystretch}{1.2}
		\begin{NiceArray}{c|c||c|c}
			\lambda & \g^{(\lambda)} & \lambda & \g^{(\lambda)}
			\\ \hline
			0 & U_i^{\ i}, \quad U_{\mu}^{\ \mu} & \epsilon_i - \epsilon_j & T_{ij}, \quad U_j^{\ j}
			\\
			\epsilon_i+\epsilon_j & T_{ij'} \quad U_i^{\ j'} & -\epsilon_i - \epsilon_j & T_{j'k}, \quad U_{j'}^{\ k}
			\\
			\delta_{\mu}-\delta_{\nu} & T_{\mu\nu}, \quad U_{\mu}^{\ \nu} & &
			\\
			\delta_{\mu}+\delta_{\nu} & T_{\mu\nu'}, \quad U_{\mu}^{\ \nu'} & -\delta_{\mu}-\delta_{\nu} & T_{\mu' \nu}, \quad U_{\mu'}^{\ \;\nu}
			\\
			2\delta_{\mu} & T_{\mu\mu'},\quad U_{\mu}^{\ \mu'} & -2\delta_{\mu} & T_{\mu'\mu},\quad U_{\mu'}^{\ \;\mu}
			\\ \hline
			\epsilon_i - \delta_{\mu} & \Lambda_{i\mu}, \quad \Gamma_i^{\ \mu} &
			-\epsilon_i+ \delta_{\mu} & \Lambda_{i'\mu'}, \quad \Gamma_{i'}^{\ \; \mu'}
			\\
			\epsilon_i + \delta_{\mu} & \Lambda_{i\mu'}, \quad \Gamma_i^{\mu'} &
			-\epsilon_i - \delta_{\mu} & \Lambda_{i'\mu}, \quad \Gamma_{i'}^{\ \;\mu}
		\end{NiceArray}
	\end{equation*}
	For $ m = 2\ell + 1$, we have additional roots
	\begin{equation*}\renewcommand{\arraystretch}{1.2}
		\begin{NiceArray}{c|c||c|c}
			\lambda & \g^{(\lambda)} & \lambda & \g^{(\lambda)}
			\\ \hline
			\epsilon_i & T_{im}, \quad U_{i}^{\ m} & -\epsilon_i & T_{i'm}, \quad U_{i'}^{\ m}
			\nn \\ \hline
			\delta_{\mu} & \Lambda_{m\mu'}, \quad \Gamma_{m}^{\ \;\mu'} & -\delta_{\mu} & \Lambda_{m\mu}, \quad \Gamma_{m}^{\ \;\mu}
		\end{NiceArray}
	\end{equation*}

Let us now turn to the normalized invariant forms on $\g$ defined by
\begin{equation*}
	\eta^{00}(X,Y) = \frac{1}{4}\operatorname{ctr}(XY),
	\qquad
	\eta^{11}(X,Y) = \frac{1}{4}\operatorname{ctr}\bigl(XM^{11}Y\bigr),
	\qquad X, Y \in \g.
\end{equation*}
The symmetry properties of these bilinear forms and the degrees $\deg(X)$, $\deg(Y)$ for which they take nonvanishing values are summarized as follows:
\begin{equation*}
	\begin{NiceArray}{c|cc|c}
		& \deg(X) & \deg(Y) & \text{symmetry}
		\\ \hline
		\Block{4-1}{\eta^{00}(X,Y)} & 00 & 00 & \Block{2-1}{\text{symmetric}}
		\\
		& 11 & 11 &
		\\ \Hline[start=2,end=4]
		& 01 & 01 & \Block{2-1}{\text{skew-symmetric}}
		\\
		& 10 & 10 &
		\\ \hline
		\Block{4-1}{\eta^{11}(X,Y)} & 00 & 11 & \Block{4-1}{\text{symmetric}}
		\\
		& 11 & 00 &
		\\
		& 01 & 10 &
		\\
		& 10 & 01 &
	\end{NiceArray}
\end{equation*}

Using the fact that the $11$-graded commutant \eqref{M11osp} can be written as
\[
	M^{11} = \sum_{i}\bigl({E_i}^i + {E_{i'}}^{i'} + {E^i}_i + {E^{i'}}_{i'}\bigr)
	-\sum_{\mu} \bigl({E_{\mu}}^{\mu} + {E_{\mu'}}^{\mu'} + {E^{\mu}}_{\mu} + {E^{\mu'}}_{\mu'}\bigr),
\]
it is easy to verify that $M^{11}$ maps the $11$- and $10$-graded basis of $\g$ to those of degree $00$ and $01$, respectively:
\begin{alignat*}{4}
&	M^{11} {U_i}^j= T_{ij}, \qquad&& M^{11} {U_i}^{j'} = T_{ij'},\qquad&& M^{11} {U_{i'}}^j= T_{i'j},&
	\notag \\
&	M^{11} {U_{\mu}}^{\nu}= -T_{\mu\nu}, \qquad && M^{11} {U_{\mu}}^{\nu'}= -T_{\mu\nu'}, \qquad && M^{11} {U_{\mu'}}^{\nu}= -T_{\mu'\nu},&
\end{alignat*}
and{\samepage
\begin{equation*}
	M^{11}{\Gamma_i}^{\mu} = \Lambda_{i\mu}, \qquad M^{11}{\Gamma_i}^{\mu'} = \Lambda_{i\mu'}, \qquad M^{11}{\Gamma_{i'}}^{\mu} = \Lambda_{i'\mu}, \qquad M^{11}{\Gamma_{i'}}^{\mu'} = \Lambda_{i'\mu'}
\end{equation*}
Therefore, the computation of $\eta^{11}$ reduces to that of $\eta^{00}$.}

Now the computation of $\eta^{00}$ and $\eta^{11}$ is straightforward. Here we list the nonvanishing components for $m = 2\ell + 1$. For $m = 2\ell$, the components corresponding to the basis elements with index $m$ do not exist. We note that additional nonvanishing components can be obtained from the symmetry properties of the bilinear forms; for simplicity, we do not list them.
Components corresponding to degree $00$ and $11$ basis:
\begin{align}
&	\eta^{00}(T_{ij},T_{pq})= \eta^{00}({U_i}^j, {U_p}^q) = \eta^{11}(T_{ij},{U_p}^q) = \delta_{iq} \delta_{jp},
	\notag\\
&	\eta^{00}(T_{ij'}, T_{p'q})= \eta^{00}\bigl({U_i}^{j'}, {U_{p'}}^{q}\bigr) = \eta^{11}(T_{ij'},{U_{p'}}^q) = \eta^{11}\bigl(T_{i'j},{U_p}^{q'}\bigr) = -\delta_{ip} \delta_{jq},
	\notag \\
	&\eta^{00}(T_{im}, T_{j'm})= \eta^{00}({U_i}^m,{U_{j'}}^m) = \eta^{11}(T_{im},{U_{j'}}^m) = \eta^{11}(T_{i'm}, {U_j}^m) = -\delta_{ij},
	\notag \\
	&\eta^{00}(T_{\mu\nu}, T_{\rho\sigma})= \eta^{00}({U_{\mu}}^{\nu},{U_{\rho}}^{\sigma}) = -\eta^{11}(T_{\mu\nu},{U_{\rho}}^{\sigma})= -\delta_{\mu\sigma} \delta_{\nu\rho},
	\notag\\
	&\eta^{00}(T_{\mu\nu'}, T_{\rho'\sigma})= \eta^{00}\bigl({U_{\mu}}^{\nu'},{U_{\rho'}}^{\sigma}\bigr) = -\eta^{11}(T_{\mu\nu'},{U_{\rho'}}^{\sigma}) =
	-\eta^{11}\bigl(T_{\mu'\nu},{U_{\rho}}^{\sigma'}\bigr)
	\notag\\
	&\hphantom{\eta^{00}(T_{\mu\nu'}, T_{\rho'\sigma})}{}
=
	\begin{cases}
		-\delta_{\mu\rho} \delta_{\nu\sigma}, & \mu < \nu, \rho < \sigma
		\\[5pt]
		-\frac{1}{2} \delta_{\mu\rho}, & \mu = \nu, \rho=\sigma,
	\end{cases}
	\label{BFosp1}
\end{align}
and $01$, $10$ basis:
\begin{align}
&	\eta^{00}(\Lambda_{i\mu}, \Lambda_{j'\nu'})= -\eta^{00}(\Lambda_{i\mu'}, \Lambda_{j'\nu}) = -\eta^{00}\bigl({\Gamma_i}^{\mu}, {\Gamma_{j'}}^{\nu'}\bigr) = \eta^{00}({\Gamma_i}^{\mu'}, {\Gamma_{j'}}^{\nu}) = \delta_{ij}\delta_{\mu\nu},
	\notag\\
&	\eta^{00}(\Lambda_{m\mu}, \Lambda_{m\nu'})= -\eta^{00}\bigl({\Gamma_m}^{\mu}, {\Gamma_m}^{\nu'}\bigr) = \delta_{\mu\nu},
	\notag\\
&	\eta^{11}\bigl(\Lambda_{i\mu},{\Gamma_{j'}}^{\nu'}\bigr)= -\eta^{11}(\Lambda_{i\mu'},{\Gamma_{j'}}^{\nu}) = \eta^{11}\bigl(\Lambda_{i'\mu},{\Gamma_{j}}^{\nu'}\bigr) = -\eta^{11}(\Lambda_{i'\mu'},{\Gamma_j}^{\nu}) = \delta_{ij} \delta_{\mu\nu},
	\notag\\
	&\eta^{11}\bigl(\Lambda_{m\mu},{\Gamma_m}^{\nu'}\bigr)=-\eta^{11}(\Lambda_{m\mu'},{\Gamma_m}^{\nu}) = \delta_{\mu\nu}.
	\label{BFosp2}
\end{align}

These bilinear forms are nondegenerate. Their inverses $\eta_{00}$ and $\eta_{11}$ are given by
\begin{align*}
	&\eta_{00}(T_{ij},T_{pq})= \eta_{00}({U_i}^j, {U_p}^q) = \eta_{11}(T_{ij},{U_p}^q) = \delta_{iq} \delta_{jp},
	\notag\\
	&\eta_{00}(T_{ij'}, T_{p'q})= \eta_{00}({U_i}^{j'}, {U_{p'}}^{q}) = \eta_{11}(T_{ij'},{U_{p'}}^q) = \eta_{11}(T_{i'j},{U_p}^{q'}) = -\delta_{ip} \delta_{jq},
	\notag \\
&	\eta_{00}(T_{im}, T_{j'm})= \eta_{00}({U_i}^m,{U_{j'}}^m) = \eta_{11}(T_{im},{U_{j'}}^m) = \eta_{11}(T_{i'm}, {U_j}^m) = -\delta_{ij},
	\notag \\
&	\eta_{00}(T_{\mu\nu}, T_{\rho\sigma})= \eta_{00}({U_{\mu}}^{\nu},{U_{\rho}}^{\sigma}) = -\eta_{11}(T_{\mu\nu},{U_{\rho}}^{\sigma})= -\delta_{\mu\sigma} \delta_{\nu\rho},
	\notag\\
&	\eta_{00}(T_{\mu\nu'}, T_{\rho'\sigma})= \eta_{00}({U_{\mu}}^{\nu'},{U_{\rho'}}^{\sigma}) = -\eta_{11}(T_{\mu\nu'},{U_{\rho'}}^{\sigma}) =
	-\eta_{11}(T_{\mu'\nu},{U_{\rho}}^{\sigma'})
	\notag\\
	&\hphantom{\eta_{00}(T_{\mu\nu'}, T_{\rho'\sigma})} {} =
	\begin{cases}
		-\delta_{\mu\rho} \delta_{\nu\sigma}, & \mu < \nu, \ \rho < \sigma,
		\\[5pt]
		-2 \delta_{\mu\rho}, & \mu = \nu, \ \rho=\sigma,
	\end{cases}
\end{align*}
and
\begin{align*}
	&\eta_{00}(\Lambda_{i\mu}, \Lambda_{j'\nu'})= -\eta_{00}(\Lambda_{i\mu'}, \Lambda_{j'\nu}) = -\eta_{00}\bigl({\Gamma_i}^{\mu}, {\Gamma_{j'}}^{\nu'}\bigr) = \eta_{00}({\Gamma_i}^{\mu'}, {\Gamma_{j'}}^{\nu}) = -\delta_{ij}\delta_{\mu\nu},
	\notag\\
&	\eta_{00}(\Lambda_{m\mu}, \Lambda_{m\nu'})= -\eta_{00}\bigl({\Gamma_m}^{\mu}, {\Gamma_m}^{\nu'}\bigr) = -\delta_{\mu\nu},
	\notag\\
&	\eta_{11}\bigl(\Lambda_{i\mu},{\Gamma_{j'}}^{\nu'}\bigr)= -\eta_{11}(\Lambda_{i\mu'},{\Gamma_{j'}}^{\nu}) = \eta_{11}\bigl(\Lambda_{i'\mu},{\Gamma_{j}}^{\nu'}\bigr) = -\eta_{11}(\Lambda_{i'\mu'},{\Gamma_j}^{\nu}) = \delta_{ij} \delta_{\mu\nu},
	\notag\\
&	\eta_{11}\bigl(\Lambda_{m\mu},{\Gamma_m}^{\nu'}\bigr)=-\eta_{11}(\Lambda_{m\mu'},{\Gamma_m}^{\nu}) = \delta_{\mu\nu}.
\end{align*}
Here again, additional nonvanishing components obtained from the symmetry properties are not listed.

We are now able to write down explicit formulas for the Casimir elements of \Zosp\ for $m = 2\ell+1$
\begin{align*}
	C_{00}={}& \sum_{i,j} \big( T_{ij}T_{ji} + {U_{i}}^{j}{U_{j}}^{i} \big) - \sum_{i<j} \big( \{T_{ij'},T_{i'j}\} + \{{U_{i}}^{j'},{U_{i'}}^{j}\} \big) \nn\\
	&-\sum_i \big( \{T_{im},T_{i'm}\} + \{{U_{i}}^m,{U_{i'}}^m\} \big)
	 - \sum_{\mu,\nu} \big( T_{\mu\nu}T_{\nu\mu} + {U_{\mu}}^{\nu}{U_{\nu}}^{\mu} \big) \nn\\
	&- 2\sum_\mu\big( \{T_{\mu\mu'},T_{\mu'\mu}\} + \{{U_{\mu}}^{\mu'},{U_{\mu'}}^{\mu}\} \big) - \sum_{\mu<\nu}\big( \{T_{\mu\nu'},T_{\mu'\nu}\} + \{{U_{\mu}}^{\nu'},{U_{\mu'}}^{\nu}\} \big) \nn\\
	&+ \sum_{i,\mu} \big( [\Lambda_{i'\mu'},\Lambda_{i\mu}] + [\Lambda_{i\mu'},\Lambda_{i'\mu}] + [{\Gamma_{i}}^{\mu},{\Gamma_{i'}}^{\mu'}] + [{\Gamma_{i'}}^{\mu},{\Gamma_{i}}^{\mu'}] \big) \nn\\
	&+ \sum_{\mu}\big( [\Lambda_{m\mu'},\Lambda_{m\mu}] + [{\Gamma_m}^{\mu},{\Gamma_m}^{\mu'}] \big),
\\
	C_{11}={}& \sum_{i,j} \{T_{ij},{U_j}^{i}\} - \sum_{i<j} \big( \{T_{ij'},{U_{i'}}^{j}\} + \{T_{i'j},{U_{i}}^{j'}\} \big) - \sum_{i}\big( \{T_{im},{U_{i'}}^m\} + \{T_{i'm},{U_{i}}^m\} \big) \nn\\
	&+ \sum_{\mu,\nu} \{T_{\mu\nu},{U_{\nu}}^{\mu}\} + 2 \sum_{\mu} \big( \{T_{\mu\mu'},{U_{\mu'}}^{\mu}\} + \{T_{\mu'\mu},{U_{\mu}}^{\mu'}\} \big) \nn\\
	&+ \sum_{\mu<\nu} \big( \{T_{\mu\nu'},{U_{\mu'}}^{\nu}\} + \{T_{\mu'\nu},{U_{\mu}}^{\nu'}\} \big) \nn\\
	&+ \sum_{i,\mu} \big( \{\Lambda_{i\mu},{\Gamma_{i'}}^{\mu'}\} - \{\Lambda_{i\mu'},{\Gamma_{i'}}^{\mu}\} + \{\Lambda_{i'\mu},{\Gamma_{i}}^{\mu'}\} - \{\Lambda_{i'\mu'},{\Gamma_{i}}^{\mu}\}\big) \nn\\
	&+ \sum_\mu \big( \{\Lambda_{m\mu},{\Gamma_{m}}^{\mu'}\} - \{\Lambda_{m\mu'},{\Gamma_{m}}^{\mu}\} \big).
\end{align*}
The corresponding Casimir elements for $m = 2\ell$ are obtained by eliminating the basis elements with index $m$.
Proposition~\ref{PROP:comm} also shows that the loop algebra of \Zosp\ admits two central extensions: one of degree $00$ and the other of degree~$11$.
Using the bilinear forms~\eqref{BFosp1} and~\eqref{BFosp2}, it is straightforward to write down all the commutation relations of the affine extension of \Zosp.
We therefore omit them here and simply note that all the relations for $\Z2$-$\osp(1|2)$ can be found in~\cite{AiSe}.

It is also worth mentioning the $\Z2$-graded color version of the super-Virasoro algebra.
In~\cite{AiSe}, by applying the Sugawara construction to the affine extension of $\Z2$-$\osp(1|2)$, a $\Z2$-graded color super-Virasoro algebra with central terms of degrees $00$ and $11$ is constructed.
This color super-Virasoro algebra contains currents of degrees $00$ and $11$, while those of degrees $01$ and $10$ are absent.
It remains a challenging problem to realize a $\Z2$-graded color super-Virasoro algebra with $00$- and $11$-graded central elements via the Sugawara construction for \Zosp.
A~realization of the color super-Virasoro algebra based on $\Z2$-$\osp(1|2)$ has also been discussed using Polyakov's soldering procedure \cite{AFITTsuper}.
In this approach, Virasoro currents of all~$\Z2$ degrees are realized, but the central extension of degree $11$ is not.
Constructing a $11$-graded central term via Polyakov's method also remains an open problem.

\section{Conclusion} \label{SEC:Conclusion}

In the present work, we have shown that color Lie algebras can admit 2nd order Casimir elements of nontrivial degrees, and that the corresponding loop algebras can admit central extensions with nontrivial gradings.
These structures arise from invariant bilinear forms on the color Lie algebras, which are determined by commutants of nontrivial degrees.

We have presented three examples of such color Lie algebras, going significantly beyond the previously known cases, in particular an infinite family of \qq\ and \Zosp.
These examples suggest the existence of a large class of color Lie algebras with similar properties.
A common feature of these examples is that they consist of multiple copies of the underlying Lie (super)algebra: \qq\ and \Zosp\ contain two copies of $\q(n)$ and $\osp(m|2n)$, respectively, while \Zsl\ consists of three copies of $\sl(2)$.

The present results give rise to an interesting and important problem, namely the physical realization of graded Casimir elements and central extensions.
It is well known that Casimir elements of Lie (super)algebras and central elements of affine Kac–Moody algebras play fundamental roles in physics.
As mentioned in the introduction, color Lie algebras are strongly related to physical systems.
Therefore, understanding the role of graded Casimir elements and central terms will be essential for applications of color Lie algebras to physical problems.

On the mathematical side, the present work reveals the rich structure of color Lie algebras compared with Lie (super)algebras.
The structure of color Lie algebras is a topic that has not yet been studied extensively.
Understanding this structure is, of course, essential for the study of representations of color Lie algebras.
In particular, the existence of zero roots in \Zosp\ suggests a rich class of highest- (or lowest-) weight representations in graded vector spaces, as investigated in~\cite{AmaAi} for $\Z2$-$\osp(1|2)$.

\subsection*{Acknowledgements}

The authors are grateful to Ruibin Zhang for helpful discussions and for providing valuable information.
J.~Van der Jeugt was supported by the Global Strategy Fund of OMU and is grateful to Osaka Metropolitan University for its hospitality. N.~Aizawa is supported by JSPS KAKENHI Grant Number JP23K03217.

\pdfbookmark[1]{References}{ref}
\LastPageEnding

\end{document}